\documentclass[11pt]{article}

\usepackage[colorlinks, bookmarks=true]{hyperref}
\usepackage[a4paper, left=2cm, right=2cm, top=2 cm, bottom= 2 cm]{geometry}
\usepackage{amsmath, amssymb, amsthm, amscd, verbatim, bbm,bm}
\usepackage[titletoc]{appendix}
\usepackage{tikz}
\usepackage{enumerate}
\usetikzlibrary{matrix,arrows}
\usepackage{pgfplots}
\usepackage{float}
\usepackage{enumitem}
\usetikzlibrary{positioning}

\usepackage{chngcntr}
\counterwithin*{equation}{section}

\def\beq{\begin{equation}}
\def\eeq{\end{equation}}

\def\beqn{\begin{equation}}
\def\eeqn{\end{equation}}

\begin{document}
\newcommand{\mb}{\mathbb}
\newcommand{\mf}{\mathfrak}
\newcommand{\mc}{\mathcal}
\newcommand{\im}{{\rm im}}
\newcommand{\del}{\partial}
\newcommand{\coker}{{\rm coker}}
\newcommand{\Tr}{\operatorname{Tr}}
\newcommand{\vol}{\operatorname{vol}}
\newcommand{\area}{\operatorname{area}}
\newcommand{\diag}{\text{diag}}
\newcommand{\f }{\mathbf }

\newcommand{\ZZ}{\mathbb{Z}}
\newcommand{\QQ}{\mathbb{Q}}
\newcommand{\RR}{\mathbb{R}}
\newcommand{\CC}{\mathbb{C}}
\newcommand{\EE}{\mathbb{E}}
\newcommand{\PP}{\mathbb{P}}
\newcommand{\NN}{\mathbb{N}}
\newcommand{\RNum}[1]{\uppercase\expandafter{\romannumeral #1\relax}}

\def\Gk{G^{(k)}}
\def\tilGk{\tilde{G}^{(k)}}
\def\sumk{\sum^{(k)}}
\def\sumi{\sum^{(i)}}
\def\ee{\mathbb{E}}
\def\T{\mathcal{T}}
\def\mfc{m_{\mathrm{fc}}}
\def\rfc{\rho_{\mathrm{fc}}}
\newcommand{\tr}{\operatorname{tr}}
\def\d{\mathrm{d}}
\def\i{\mathrm{i}}
\def\e{\mathrm{e}}
\def\eps{\varepsilon}
\def\DLeps{\mathcal{D}_{L, \varepsilon_0}}
\def\Hk{H^{(k)}}
\def\tilHk{\tilde{H}^{(k)}}
\def\mk{m^{(k)}_N}
\def\tilmk{\tilde{m}^{(k)}_N}
\def\Im{\mathrm{Im}}
\def\Xik{\Xi^{(k)}}
\def\1{\mathbbm{1}}
\def\G{\mathcal{G}}
\def\E{\mathcal{E}}
\def\tilbarGk{\bar{\tilde{G}^{(k)}}}
\def\P{\mathcal{P}}
\def\V{\mathcal{V}}
\def\tilh{\tilde{h}}
\def\bolde{\boldsymbol{e}}
\def\Id{\mathrm{Id}}
\def\rhosc{\rho_{\mathrm{sc}}}
\def\rhofc{\rho_{\mathrm{fc}}}
\def\thetat{\vartheta_t }
\def\nn{\mathbb{N}}
\def\F{\mathcal{F}}
\def\A{\mathcal{A}}
\def\D{\mathcal{D}}
\def\sc{\mathrm{sc}}
\def\delxi{\delta_\xi}
\def\pp{\mathbb{P}}
\def\zz{\mathbb{Z}}
\def\hatG{\hat{G}}
\def\hatdel{\hat{\del}}
\def\rr{\mathbb{R}}
\def\hatH{\hat{H}}
\def\fc{\mathrm{fc}}
\def\N{\mathcal{N}}
\def\hatA{\hat{A}}
\def\aq{\alpha_q}
\def\DLz{\mathcal{D}_L^{(0)}}
\def\DLo{\mathcal{D}_L^{(1)}}
\def\DLt{\mathcal{D}_L^{(2)}}
\def\mN{m_N}
\def\Re{\mathrm{Re}}

\theoremstyle{plain}
  \newtheorem{theorem}{Theorem}[section]
  \newtheorem{proposition}[theorem]{Proposition}
  \newtheorem*{propositionnonumber}{Proposition}
  \newtheorem{lemma}[theorem]{Lemma}
  \newtheorem{corollary}[theorem]{Corollary}
  \newtheorem{conjecture}[theorem]{Conjecture}
\theoremstyle{definition}
  \newtheorem{assumption}[theorem]{Assumption}
  \newtheorem{definition}[theorem]{Definition}
  \newtheorem{example}[theorem]{Example}
  \newtheorem{examples}[theorem]{Examples}
  \newtheorem{question}[theorem]{Question}
  \newtheorem{problem}[theorem]{Problem}
  \newtheorem{remark}[theorem]{Remark}
  
  \title{Graph Laplacian}
  \author{}

\begin{table}
\centering

\begin{tabular}{c}
\multicolumn{1}{c}{\Large{\bf Spectral statistics of sparse Erd\H{o}s-R\'enyi graph Laplacians}}\\
\\
\\
\end{tabular}
\begin{tabular}{c c c}
Jiaoyang Huang & & Benjamin Landon\\
\\
 \multicolumn{3}{c}{ \small{Department of Mathematics} } \\
 \multicolumn{3}{c}{ \small{Harvard University} } \\
\small{jiaoyang@math.harvard.edu} &  & \small{landon@math.harvard.edu}  \\
\\
\end{tabular}
\\
\begin{tabular}{c}
\multicolumn{1}{c}{\today}\\
\\
\end{tabular}

\begin{tabular}{p{15 cm}}
\small{{\bf Abstract:} We consider the bulk eigenvalue statistics of Laplacian matrices of large Erd\H{o}s-R\'enyi random graphs in the regime $p \geq N^{\delta}/N$ for any fixed $\delta >0$.  We prove a local law down to the optimal scale $\eta \gtrsim N^{-1}$ which implies that the eigenvectors are delocalized.   We consider the local eigenvalue statistics and prove that both the gap statistics and averaged correlation functions coincide with the GOE in the bulk.  }
\end{tabular}
\end{table}

\section{Introduction}

{\let\thefootnote\relax\footnote{The work of B.L. is partially supported by NSERC. }} In the late 1950s Erd\H{o}s and R\'enyi introduced what is perhaps the simplest random graph model, the now well-known Erd\H{o}s-R\'enyi random graph $\mathcal{G} (N, p)$. 
This model is a graph on $N$ vertices formed by choosing each (undirected) edge $(i, j)$ independently with probability $p$. 
We will say that the model is sparse when $p \to 0$ as $N \to \infty$ as each vertex has on average $pN \ll N$ edges.

Random graphs have numerous applications in the modeling of complex physical, biological and social systems.  A common method to understand the structure of the underlying graph 
is to associate to the graph a matrix, the most intensively studied being the adjacency and Laplacian matrices.
The spectral properties of these matrices carry a wealth of information about the underlying graph.  For example, the eigenvalues of the adjacency matrix reflect the topological features of the corresponding graphs, such as connectivity, while the eigenvalues and eigenvectors of the Laplacian matrix can be used to study random walks on the graph.  For a comprehensive introduction to the spectral properties of random graphs we refer the interested reader to \cite{CFSpectral}.  

Our main interest in this paper is the spectral properties of the Laplacian of the Erd\H{o}s-R\'enyi random graph.  There is an enormous literature on the study of the Laplacian of general random graphs and we do not attempt to give an exhaustive list of existing work, only mentioning some results on spectral properties related to ours.
The limiting spectral measure of the non-sparse Erd\H{o}s-R\'enyi Laplacian was first derived in \cite{BDJHankel}, and the sparse case was obtained in \cite{JTDilute}.  
Laws of large numbers of extreme eigenvalues of the Erd\H{o}s-R\'enyi Laplacian were proved in \cite{JTLowEv, DJLowEv} in the non-sparse case. For the spectral properties of the Laplacian matrices of more general random graphs, we refer the reader to, e.g., \cite{CRSpectralGeneral, LPNormBound,CLVPowerLaw, CLVExpectedD, COEigenvalue}.

Random matrix theory provides a natural framework for studying the spectral properties of random graphs.  
One of the most well-studied models is the Wigner ensemble consisting of matrices with iid entries, (up to the symmetry constraint $H=H^*$) introduced by Wigner \cite{wigner} to model the spectra of heavy ions.   The Wigner-Dyson-Gaudin-Mehta conjecture asserts that the local spectral statistics of Wigner matrices exhibit universality and are independent of the underlying distribution of the matrix entries, and instead depend only on the symmetry class (real symmetric or complex Hermitian) of the ensemble.  In particular, the local statistics coincide with the case in which the entries are Gaussians for which there are explicit formulas.  This conjecture has recently been solved for all symmetry classes in  \cite{EPRSYBulkUniv, ERSTVYBulkUniv, ESYRelaxation,  EYGap, EYYBulkGen, BEYYHomog}  (parallel results in certain cases were obtained in \cite{tao2010, tao2011}).  

The adjacency matrix of a non-sparse Erd\H{o}s-R\'enyi graph (i.e., $p$ fixed independent of $N$) essentially falls into the Wigner class and so the results  \cite{EPRSYBulkUniv, ERSTVYBulkUniv, ESYRelaxation,  EYGap, EYYBulkGen, BEYYHomog}  provide detailed information of the local spectral statistics of this ensemble.  The sparse case can be viewed as a singular Wigner ensemble, as the law of the entries is highly concentrated around $0$, and its analysis is much harder than the non-sparse case.  Universality of the local statistics of the sparse Erd\H{o}s-R\'enyi adjacency matrix was achieved in the regime $p \geq N^{2/3+\delta}/N$ in the works  \cite{EKYYSparse1, EKYYSparse2}; universality (of the bulk statistics) in the regime $p \geq N^{\delta}/N$ was only achieved quite recently in \cite{Sparse}.  

A key property of the Wigner ensemble is that all of the entries are independent (up to the symmetry constraint $H=H^*$).  In this regard, the new major challenge in the analysis of the Erd\H{o}s-R\'enyi Laplacian is the correlation between the entries, an obstacle which is present even in the non-sparse case.  To our knowledge, detailed information about local spectral statistics of (non-invariant) matrix ensembles with correlated entries has been limited to only a few cases.  A local law was proved for sparse $d$-regular graphs in \cite{DPDRegular} down to the scale $\eta \gtrsim \log(N)^{-1}$.  A local law down to the optimal scale was achieved for the $d$-regular random graph in the regime $d \gtrsim (\log N)^4$ in the work \cite{BKYDRegular} and universality of the local spectral statistics was subsequently obtained in the regime $d \gtrsim N^{\eps}$ in \cite{BHKYdregular}.  A local law was obtained for matrices having with correlated entries having a certain four-fold symmetry in \cite{altfourfold} and universality for Gaussian matrices with a translation invariant correlation structure was obtained in \cite{corgauss}.

Our first contribution is to prove a local law for the Erd\H{o}s-R\'enyi Laplacian down to the optimal scale.  More precisely we prove that the Stieltjes transform of the empirical eigenvalue distribution is well-approximated by the free convolution of the semicircle law and a standard Gaussian down to the scale $\eta \gtrsim N^{-1}$.  This implies a rigidity estimate for any fixed fraction of the non-extremal eigenvalues and delocalization of the corresponding eigenvectors.  Our proof of the local law is inspired by methods developed for symmetric matrices with uncorrelated entries \cite{ESYSemicircle, ESYLocalLaw, EYYrig, EKYYSparse1,EKYYgeneral, LSdeformW}.  However, these proofs rely on the fact that the $i$th row and column of the matrix is independent of $i$th minor.  This difficulty was overcome in  the work \cite{BKYDRegular} on the $d$-regular graph using methods exploiting the specific structure of the underlying random graph,  and it is not clear that these methods would be applicable to the general class of models we will consider in this paper.  We instead use the fact that the correlation between the matrix entries is quite weak and implement a high-order resolvent expansion in order to estimate high moments of certain error terms arising in the proof.  We introduce a graphical notation to keep track of terms appearing in our estimates inspired by the methods of \cite{iso2014}.  Our expansion shows that the correlation between the $i$th minor and the $i$th row can be ignored.

The solution of the Wigner-Dyson-Gaudin-Mehta conjecture saw the development of a general three-step strategy for proving bulk universality for random matrix ensembles.  The first step is to prove a local law controlling the eigenvalue density down to the optimal scale.  The second step concerns the proof of the optimal time to local equilibrium of Dyson Brownian motion (DBM), and the third step is either a comparison or perturbative argument comparing the random matrix ensemble to DBM.  Our result discussed above completes the first step.

In the second step we introduce a modification of the DBM, defining a matrix process that preserves the structure of the Laplacian matrix.  Note that the Laplacian matrix has a trivial eigenvalue corresponding to the constant eigenvector.  Once we restrict the matrix process to the space orthogonal to this eigenvector, we get an ensemble that is well approximated by the sum of a diagonal matrix and a GOE matrix.  Universality for this process then follows from the recent work \cite{DBM}, where the optimal time to local equilibrium was established for a wide class of initial data; we remark that the result \cite{DBM} is required to deal with the sparse ensembles considered here, and previous approaches along the lines of \cite{ESYRelaxation, LSSYdeformWB} are insufficient (see \cite{ES} for related results on DBM with general initial data).  In order to complete the third step we rely on the observation in \cite{BoYa} that the change of the eigenvalues is negligible for a short time if one views DBM as a matrix process.   As a result we obtain bulk universality for the Erd\H{o}s-R\'enyi Laplacian in the sense of gap universality and averaged correlation functions in the regime $p \geq N^{\delta}/N$.  

Our analysis applies to a wider class of random matrices that includes the Erd\H{o}s-R\'enyi Laplacian as a specific example.  In Section \ref{defmodel} we introduce this model which we call the random Laplacian-type matrix and state our main results.  In Section \ref{locallawsec} we prove the local law for random Laplacian-type matrices.  In Section \ref{univsec} we prove universality of the local statistics.

\vspace{5 pt}
\noindent{\bf Acknowledgements.} Part of the work of B.L. was conducted during a stay at the Erwin Schr\"odinger Institute. B.L. would like to thank ESI for their hospitality.  The authors are grateful to Phillippe Sosoe for helpful comments on a draft of this paper.

\section{Definition of model and main results}\label{defmodel}

\subsection{Definition of model}

In this section we introduce the class of random matrices under consideration which we call random Laplacian-type matrices. The motivating example is the graph Laplacian of the Erd\H{o}s-R\'enyi random graph. 
The Erd\H{o}s-R\'enyi  graph on $N+1$ vertices is constructed by choosing each undirected edge $(i, j)$, $i \neq j$ independently with probability $p$ (for reasons that will become clear later it is more natural to consider a graph on $N+1$ vertices instead of $N$ vertices.  As a result our matrices will be of size $N+1 \times N+1$).  It will be convenient to introduce the {\bf sparsity parameter} $q$ by
\beqn
p := \frac{q^2}{N}
\eeqn
We denote the Laplacian of the Erd\H{o}s-R\'enyi random graph by $M= (M_{ij})_{ij}$. The off-diagonal entries are given by
\beqn
M_{ij} = m_{ij}, \qquad  1\leq i \neq j\leq N+1
\eeqn
where the $m_{ij}$ are independent Bernoulli random variables (up to the symmetry constraint $m_{ij} = m_{ji}$) satisfying
\beqn
\PP\left(m_{ij}=1\right)=\frac{q^2}{N},\qquad \PP(m_{ij} = 0)=1-\frac{q^2}{N}.
\eeqn
The diagonal entries of $M$ are given by
\beqn
 M_{ii} = - \sum_{1\leq k\leq N+1, k \neq i} m_{ki}.
\eeqn
Note that on average each row has $pN = q^2$ nonzero elements and so the matrix is sparse in the regime $q = o(N^{1/2})$. In order to study the spectral statistics of $M$ we will extract the mean and introduce a rescaling so that the typical eigenvalue spacing in the bulk is of the order $N^{-1}$.  We write
\begin{align} \label{eqn:mean}
M = q(1-q^2/N)^{1/2}H + q^2 \f e \f e^* - q^2(N+1) I_{N+1}.
\end{align}
Here, we have introduced the $(N+1)\times 1$ column vector
\beq\label{theev}
\f e =\left((N+1)^{-1/2}, (N+1)^{-1/2},\cdots, (N+1)^{-1/2}\right)^*.
\eeq
The matrix $H = (H_{ij})$ is of the same form as $M$,
\beqn
H_{ij} = h_{ij}, \quad 1\leq i \neq j\leq N+1, \qquad H_{ii} = -\sum_{1\leq k\leq N+1, k \neq i} h_{ki},
\eeqn
where the $h_{ij}$ are independent random variables (up to the symmetry constraint $h_{ij} = h_{ji}$) defined by
\beqn
h_{ij}:=\frac{m_{ij}}{q ( 1 - q^2/N)^{1/2}}-\frac{q}{N ( 1 - q^2/N)^{1/2}}, \qquad 1\leq i \neq j\leq N+1.
\eeqn
Note that the entries of $H$ are centered and for any indices $1\leq i\neq j \leq N+1$, the off-diagonal element $h_{ij}$ satisfies
\beqn
\ee [ |h_{ij} |^2 ] = \frac{1}{N}, \qquad \ee [ |h_{ij} |^k] \leq \frac{1}{N q^{k-2}}, \qquad k \geq 3.
\eeqn
\begin{remark}
Note that since $\f e$ is an eigenvector of $H$, all spectral properties of $M$ can be deduced from $H$ as defined in \eqref{eqn:mean}.  We will therefore state and prove results only for $H$. Analogous statements for $M$ follow trivially.
\end{remark}
The above discussion leads us to the following definition of random Laplacian-type matrices.
\begin{definition}\label{def:H} A real symmetric matrix $H = (H_{ij} )_{ij}$ is called a {\bf random Laplacian-type matrix} if it is of the form
\begin{align}\label{e:LM}
H=\left(
\begin{array}{cccc}
-\sum_{j\neq1}h_{1j} & h_{12} &\cdots & h_{1{N+1}}\\
h_{21} & -\sum_{j\neq 2}h_{2j} & \cdots &h_{2{N+1}}\\
\vdots & \vdots & \vdots & \vdots\\
h_{{N+1}1} & h_{N+12} & \cdots & -\sum_{j\neq N+1}h_{N+1j}
\end{array}
\right)
\end{align}
where $h_{ij}$ are independent random variables (up to the symmetry constraint $h_{ij} = h_{ji}$) with mean zero and variance $1/N$. We assume that they satisfy the moment conditions: for any $p\geq 3$, there exists some constant $c_p$ such that 
\begin{align}\label{asup:moment}
\EE[|h_{ij}|^p]\leq \frac{c_p}{q^{p-2}N},
\end{align}
where the {\bf sparsity parameter} $q$ satisfies
\begin{equation}\label{defq}
N^{\beta}\leq q\leq N^{1/2}
\end{equation}
for some $\beta >0$.
\end{definition}

\remark We state and prove our results only in the real symmetric case, however our results and proofs carry over to the complex Hermitian case with only minor notational changes.  

\subsection{Free convolution}

The results of \cite{BDJHankel, JTDilute} state that the empirical eigenvalue distribution of the random Laplacian-type matrix converges weakly to the free convolution of the semicircle law and standard Gaussian, which we define in this section.
We denote this free convolution by $\rhofc$,
\beq \label{eqn:freeconv}
\rhofc = \rhosc \boxplus \rho_G.
\eeq
Above, $\rhosc$ and $\rho_G$ are the semicircle and standard Gaussian distributions, respectively, and are given by 
\beqn
\rhosc (E) = \frac{1}{2 \pi} \sqrt{ ( 4 - E^2)_+ },\qquad 
\rho_G (E) = \frac{1}{ \sqrt{2 \pi } } \e^{- E^2/2}.
\eeqn
The precise definition of \eqref{eqn:freeconv} is through its Stieltjes transform which we denote by $\mfc$.  The Stieltjes transform $\mfc$ is defined as the solution to the following functional equation on the upper half complex plane $\Im[z]>0$,
\begin{equation}\label{e:mfc}
\mfc (z) = \int \frac{\rho_G(x)}{ x - z - \mfc (z)} \d x.
\end{equation}
 It is well known \cite{Biane} that there is a unique solution satisfying $\Im[\mfc(z)]\geq 0$ and $\mfc(z)\sim z^{-1}$ as $z\rightarrow \infty$, and its properties are well-studied.  In particular, $\mfc$ is the Stieltjes transform of a measure with density $\rhofc$ that is analytic and can be recovered by the Stieltjes inversion formula
\beqn
\rhofc (E) = \lim_{\eta \downarrow 0} \frac{1}{ \pi} \Im [ \mfc (E + \i \eta ) ].
\eeqn

We denote the classical eigenvalues of $\rhofc$ and $\rhosc$ by $\gamma_i$ and  $\gamma_{i,sc}$ respectively.  They are defined by the equations
\begin{align}\label{e:classicallocal}
\int_{-\infty}^{\gamma_i}\rhofc(x)dx=\frac{i-\frac{1}{2}}{N},\qquad
\int_{-\infty}^{\gamma_{i,sc}}\rhosc(x)dx=\frac{i-\frac{1}{2}}{N}. 
\end{align}

\subsection{Main results}
A random Laplacian-type matrix $H$ has $N+1$ eigenvalues, one of which is always $0$ with eigenvector $\f e$.  We call this the trivial eigenvalue.  We denote the remaining $N$ nontrivial eigenvales of $H$ as $\lambda_1(H)\leq \lambda_2(H)\leq \cdots \leq \lambda_{N}(H)$ and corresponding eigenvectors as  $u_1(H), u_2(H),\cdots ,u_{N}(H)$. We will sometimes write them as $\lambda_i$ and $u_i$ for simplicity of notation.  


In this paper we are interested in the properties of the ``bulk'' eigenvalues of $H$.  We fix a large constant $L>0$ and consider eigenvalues in the interval $[-L, L]$.  We fix the parameter 
\beq\label{e:defineXi}
\xi :=N^{\nu}
\eeq
where $\nu>0$ is a small parameter.  The parameter $\nu >0$ can be arbitrarily small, but in proofs its value may change from line to line; this will only occur at finitely many steps and we will not track its value explicitly.

 We define the spectral domain $\D_L=\DLz\cup \DLo\cup \DLt$, where
\begin{align}
\DLz=&\{ z =E+\i\eta : |E| \leq 2\xi+L, 2 \leq \eta \leq L \}, \notag \\
\DLo=&\{ z = E + \i \eta : |E| \leq L, \xi^3N^{-1} \leq \eta \leq 2 \}, \notag \\
\DLt=&\{ z = E + \i \eta : 2\xi \leq |E| \leq 2\xi + L ,  \xi^3N^{-1} \leq \eta \leq 2 \}.
\end{align}
On the domain $\DLo$ we will prove a local law down to the optimal scale.  In order to obtain information on the eigenvalue locations in the interval $[-L, L]$ we will also require estimates on the domains $\DLz$ and $\DLt$.  

The spectral parameter $z$ will play a fundamental role in this paper.  We will write it as the sum of its real and imaginary parts as
\beqn
z = E + \i \eta, \qquad E = \Re[z],\quad \eta = \Im [z]
\eeqn
We define the control parameter
\beq\label{controlPsi}
\Psi (z) = \frac{1}{q} + \frac{1}{ ( N \eta )^{1/2} }.
\eeq
On the domain $\D_L$ we always have 
\beqn
\xi\Psi\leq \log (N)^{-1}
\eeqn
 for $N$ large enough.

Throughout the paper we use the following notion of {\bf overwhelming probability}.
\begin{definition}  We say that a family of events $\F (u)$ indexed by some parameter(s) $u$ holds with overwhelming probability if for any large $D>0$ we have for all $N$ large enough,
\beqn
\pp[ \F (u) ] \geq 1 - N^{-D},
\eeqn
 uniformly in $u$.

Given two families of events $\F_1 (u)$ and $\F_2 (u)$ we say that $\F_2$ holds with overwhelming probability on $\F_1$ if for any large $D>0$, we have for all $N$ large enough, $\pp[ \F_2 (u) ^c \cap \F_1 (u) ] \leq N^{-D}$, uniformly in $u$.

Often we will take $u$ to lie in the spectral domain $\D_L$ or the indices $[[1, N+1]]$ or some combination thereof.  Unless it is mentioned otherwise, the notion of overwhelming probability will always be uniform in the parameters under consideration.
\end{definition}

We denote the resolvent of the matrix $H$ by
\beqn
G (z) := (H-z)^{-1},
\eeqn
and the Stieltjes transform of the empirical eigenvalue distribution of $H$ by
\beqn
\mN(z):=\frac{1}{N}\Tr G(z).
\eeqn
Sometimes we will omit the parameter $z$ and simply denote them by $G$ and $\mN$, if the context is clear. 

Now we can state the main results of this paper.
\begin{theorem}[Local Law]\label{t:lclaw} Let $H$ be a random Laplacian-type matrix as in Definition \ref{def:H} and let $G(z)$ be its resolvent.  There is a constant $C>0$ such that with overwhelming probability, 
\begin{align}\label{e:locallaw}
 |\mN(z)-\mfc(z)| \leq C\xi\Psi 
\end{align}
and
\begin{align}\label{e:etb}
 \left| G_{ij} (z) - \delta_{ij} \frac{1}{ \sum_{k \neq i} h_{ik} - z - \mfc (z) } \right|  \leq  C\xi \Psi 
\end{align}
uniformly for $z\in \D_L$.
\end{theorem}

Theorem \ref{t:lclaw} implies the following rigidity estimates on the eigenvalues and complete delocalization of eigenvectors.
\begin{corollary}[Rigidity of Eigenvalues] \label{cor:rig} Let $H$ be a random Laplacian-type matrix as in Definition \ref{def:H}. 
Fix a small $\kappa >0$.  With overwhelming probability,
\beqn
|\lambda_i(H)-\gamma_i| \leq C \xi^2 q^{-1},
\eeqn
uniformly for indices $i\in [[\kappa N, (1-\kappa)N]]$.
\end{corollary}

\begin{corollary}[Delocalization of Eigenvectors]\label{cor:del}
Let $H$ be a random Laplacian-type matrix as in Definition \ref{def:H}. There is a constant $C>0$ so that the following holds.  Fix a small $\kappa>0$.  With overwhelming probability we have
\beqn
\|u_i(H)\|^2_\infty \leq C \xi^3 N^{-1},
\eeqn
uniformly for indices $i\in [[\kappa N, (1-\kappa)N]]$.
\end{corollary}

\begin{remark}
In this note, we focus on the bulk eigenvalues of the random Laplacian-type matrix. As for the behavior of extreme eigenvalues, we refer the interested reader to the papers \cite{DJLowEv, JTLowEv}, where it is proved that for the non-sparse random Laplacian-type matrix (i.e. $q=N^{1/2}$) that for any fixed index $i$,
\beqn
\frac{\lambda_i(H)}{\sqrt{2N\log N}}\longrightarrow 1,
\eeqn
almost surely.  The eigenvectors corresponding to those extreme eigenvalues can be studied in the Gaussian case \eqref{e:GLM} by finite rank perturbation theory \cite[Theorem 2.2]{BN}. 
We cannot expect delocalization results such as Corollary \ref{cor:del}.
\end{remark}

Finally, we have the following result on the bulk universality of the eigenvalue statistics.  To state it we define the $n$-point correlation functions.  Denote by
\beq
p_H ( \lambda_1, \cdots \lambda_N )
\eeq
the unordered eigenvalue density of the matrix $H$.  The $n$-point correlation function is defined by
\beq
p^{(n)}_H ( \lambda_1, \cdots \lambda_n ) = \int p_H ( \lambda_1, \cdots \lambda_N ) \d \lambda_{n+1} \cdots \d \lambda_N.
\eeq
We denote the corresponding quantities for the GOE by $p_{GOE}^{(n)}$ and $p_{GOE}$.
\begin{theorem}\label{t:uni}
Let $H$ be a random Laplacian-type matrix as in Definition \ref{def:H}.  Then the nontrivial eigenvalues of $H$ exhibit universality in the following two forms. Firstly, the nontrivial eigenvalues of $H$ have gap universality. For any $\kappa>0$ and index $i\in [[\kappa N, (1-\kappa)N]]$
\beqn
\begin{split}
\lim_{N\rightarrow \infty}&\EE^{(H)}[ O(N \rfc(\gamma_i)(\lambda_i-\lambda_{i+1}), \cdots, N\rfc(\gamma_i)(\lambda_i -\lambda_{i+n}))]\\
-&\EE^{(GOE)}[ O(N \rho_{sc}(\gamma_{i, \sc})(\lambda_i-\lambda_{i+1}), \cdots, N\rho_{sc}(\gamma_{i, \sc})(\lambda_i -\lambda_{i+n}))]=0,
\end{split}
\eeqn
for test functions $O \in C_0^\infty ( \rr^n)$, where $\gamma_{i,\sc}$'s are the classical eigenvalue locations with respect to semi-circle law $\rho_{sc}$.

Secondly, the averaged $n$-point correlation functions of the nontrivial eigenvalues of $H$ are universal. For any $\delta >0$, $E\in [-L,L]$,  and $b\geq N^{-1+\delta}$ we have
\beqn
\begin{split}
\lim_{N\rightarrow \infty}\int_{E-b}^{E+b} \int_{\RR^n} O(\alpha_1,\cdots, \alpha_n)\left\{ \frac{1}{\rfc(E)^n}\rho^{(n)}_{H}\left( E'+\frac{\alpha_1}{N\rfc(E)},\cdots E'+\frac{\alpha_n}{N\rfc(E)}\right)\right.\\
-\left.\frac{1}{\rho_{sc}(E)^n}\rho^{(n)}_{GOE}\left( E'+\frac{\alpha_1}{N\rho_{sc}(E)},\cdots E'+\frac{\alpha_n}{N\rho_{sc}(E)}\right)\right\} \d \alpha_1 ... \d \alpha_n \frac{\d E'}{2b}=0.
\end{split}
\eeqn
where the test function $O\in C^{\infty}_{c}(\RR^n)$.
\end{theorem}

\subsection{Toy Gaussian Model}\label{s:toy}
In this section we introduce a toy Gaussian model of the random Laplacian-type matrix.  We let $W$ be the following $(N+1) \times (N+1)$ real symmetric matrix
\begin{align}\label{e:GLM}
W=\left(
\begin{array}{cccc}
-\sum_{j\neq1}w_{1j} & w_{12} &\cdots & w_{1N+1}\\
w_{21} & -\sum_{j\neq 2}w_{2j} & \cdots &w_{2N+1}\\
\vdots & \vdots & \vdots & \vdots\\
w_{N+11} & w_{N+12} & \cdots & -\sum_{j\neq N+1}w_{N+1j}
\end{array}
\right)
\end{align}
where $\{w_{ij}, 1\leq i< j \leq N+1\}$ are independent standard Gaussian random variables with mean $0$ and variance $1/N$.

We sketch a proof of Theorem \ref{t:lclaw} and \ref{t:uni} for the above toy model. This argument heavily depends on the Gaussian structure and cannot be generalized to the random Laplacian-type matrix of Definiton \ref{def:H}. However, the argument provides some intuition as to why one expects that Theorem \ref{t:uni} holds.  More importantly, the structure of $W$ exhibited in the decomposition \eqref{e:declaw} will play a role in the proof for the general case.

The matrix $W$ always has a trivial eigenvalue $0$ with eigenvector $\f e$ and so we need to remove this eigenvalue in order to obtain universality.
 We project $W$ down to an $N\times N$ matrix by conjugation with an orthogonal matrix $R$. Let $\tilde{R}$ be any fixed $(N+1)\times (N+1)$ deterministic orthogonal matrix with last column $\f e$.  That is, $\tilde{R}=(R| \f e)$, where $R$ is the $(N+1)\times N$ matrix consisting of the first $N$ columns of $\tilde{R}$. We then have
\beqn
\tilde{R}^*W\tilde{R}=\hat{W}\oplus 0,
\eeqn
where $\hat{W}$ is the $N \times N$ matrix
\beqn
\hat{W} := R^* W R.
\eeqn
The entries of the matrix $\hat{W}$ are joint Gaussian random variables. We compute the covariance structure of the matrix $\hat{W}$, which will prove the following.
\begin{proposition}\label{p:decom_G}
We have the following equality in law,
\begin{align}\label{e:declaw}
\hat{W}\stackrel{d}{=}GOE+ R^* D R+gI_{N},
\end{align}
where $GOE$ is an $N\times N$ GOE matrix, 
 $D$ is an $(N+1)\times (N+1)$ diagonal matrix with diagonal entries independent Gaussian random variables with variance $(N+1)/N$, and $g$ is a Gaussian random variable with variance $1/N$.  All three are independent of each other.
\end{proposition}
\begin{proof} The matrix elements of both sides of \eqref{e:declaw} are jointly distributed centered Gaussians.  
Hence we need only to check that the covariances of the matrix elements are the same. For the left hand side of \eqref{e:declaw} we first compute the covariance structure of $W$. We introduce auxiliary independent Gaussian variables $w_{ii}$'s, for $1\leq i\leq N+1$, with variance $2/N$.  We have,
\begin{align}
\EE[W_{ij}W_{kl}]
=&\EE[(w_{ij}-\delta_{ij}\sum_{m=1}^{N+1} w_{mi})(w_{kl}-\delta_{kl}\sum_{m=1}^{N+1} w_{mk})]\notag\\
=&\frac{1}{N}\left(\delta_{ik}\delta_{jl}+\delta_{il}\delta_{jk}-\delta_{ijk}-\delta_{ijl}-\delta_{ikl}-\delta_{jkl}+\delta_{ij}\delta_{kl}+(N+1)\delta_{ijkl}\right)
\end{align}
where $\delta_{abc}=1$ (resp $\delta_{abcd}=1$) if and only if $a=b=c$ (resp $a=b=c=d$), otherwise it is zero. 
For indices $a, b, c, d \in [[1, N]]$ we have
\begin{align}
&\EE[\hat{W}_{ab}\hat{W}_{cd}]=\EE[\sum_{i,j,k,l}R_{ia}W_{ij}R_{jb}R_{kc}W_{kl}R_{ld}]\notag\\
=&\frac{1}{N}\sum_{i,j,k,l}R_{ia}R_{jb}R_{kc}R_{ld}\left(\delta_{ik}\delta_{jl}+\delta_{il}\delta_{jk}-\delta_{ijk}-\delta_{ijl}-\delta_{ikl}-\delta_{jkl}+\delta_{ij}\delta_{kl}+(N+1)\delta_{ijkl}\right)\notag\\
=&\frac{1}{N}\left(\delta_{ac}\delta_{bd}+\delta_{ad}\delta_{bc}+\delta_{ab}\delta_{cd}+(N+1)\sum_{i=1}^{N+1}R_{ia}R_{ib}R_{ic}R_{id}\right),
\end{align}
where we used the fact that all columns of $R$ are orthogonal to $\f e$.
For the righthand side of \eqref{e:declaw} we compute
\begin{align}
&\EE[(GOE+  R^*D R+gI_{N})_{ab} \times (GOE+  R^*D R+gI_{N})_{cd}]\notag\\
=&\EE[GOE_{ab} GOE_{cd}]+ \EE[(R^* D R)_{ab} (R^* DR)_{cd}]+\EE[(gI_{N})_{ab}  (gI_{N})_{cd}]\notag \\
=&\frac{1}{N}(\delta_{ac}\delta_{bd}+\delta_{ad}\delta_{bc})+
\frac{N+1}{N}\sum_{i=1}^{N+1}R_{ia}R_{ib}R_{ic}R_{id}+\frac{1}{N}\delta_{ab}\delta_{cd}
\end{align}
This finishes the proof.
\end{proof}

The advantage of the decomposition \eqref{e:declaw} is that the two summands $R^* DR+gI_{N}$ and $GOE$ are independent. In this case, both the global empirical eigenvalue distribution and the local eigenvalue statistics are well-understood. For the global eigenvalue distribution, one notices that the empirical distribution of $R^* D R+gI_{N}$ converges to the law of a standard Gaussian. From the theory of free probability, the global law of eigenvalue distribution of $\hat{W}$ converges to free convolution of semi-circle and standard gaussian distributions. We refer to \cite[Chapter 5]{AGZ} for an introduction on free probability. For the local eigenvalue statistics, one can show that 
 the Stieltjes transform of the empirical eigenvalue distribution of $R^* D R+g I_N$ is close to that of a Gaussian down to the optimal scale, i.e. for $\eta \gtrsim N^{-1}$,  with overwhelming probability. 
 The RHS of \eqref{e:declaw} is a deformed Wigner ensemble, and universality of the bulk eigenvalue statistics can be concluded from the main results of \cite{LSSYdeformWB} or \cite{DBM}. 


\section{Local laws}\label{locallawsec}
The goal of this section is to prove Theorem \ref{t:lclaw}, the local law for random Laplacian-type matrices. The proof is based on the analysis of the self-consistent equation for the empirical Stieltjes transform of $H$. The main difficulty here is the  correlation between diagonal entries and off-diagonal entries of $H$. As a result, the analysis is more involved than in the case of Wigner or deformed Wigner ensembles.

\subsection{Notation and Prelimilaries}
We introduce some notation. We denote the $k$-th minor of $H$ by $\Hk$, which is the submatrix of $H$ with the $k$-th row and column removed. We denote the random Laplacian-type matrix of the subgraph obtained by removing the $k$-th vertex by $\tilHk$.  More precisely, we define $\tilHk$ by
\beqn
\Hk_{ij} =: \tilHk_{ij} - \delta_{ij} h_{ik}, \qquad i, j \neq k.
\eeqn
We denote the Green's functions of $H, \Hk$ and $\tilHk$ by $G$, $\Gk$ and $\tilGk$ respectively.  We furthermore define the normalized Stieltjes transforms by
\beqn
\mN(z) = \frac{1}{N} \tr G(z), \quad \mk (z) = \frac{1}{N} \tr \Gk (z) , \quad \tilmk (z) = \frac{1}{N} \tr \tilGk (z).
\eeqn
Note that we normalize the Stieltjes transform of $H$ by $1/N$ even though it is an $(N+1) \times (N+1)$ matrix. 
Finally, we denote by 
\beqn
\sumk_i = \sum_{1\leq i\leq N+1\atop i\neq k},\quad \sumk_{ij}=\sum_{1\leq i,j\leq N+1\atop i,j\neq k}
\eeqn
summation in the index $i$ or $i,j$ omitting the index $k$.

In the following we collect some resolvent identities which will be used repeatedly in the paper.
\begin{proposition} The Schur complement formula
\begin{equation} \label{eqn:schur}
G_{kk}(z) = \frac{1}{ -\sumk_i h_{ik} - z- \sumk_{ij} h_{ki}\Gk_{ij} h_{kj} }.
\end{equation}
The identities for $i, j \neq k$
\begin{equation} \label{eqn:minorid}
G_{ij} = \Gk_{ij} + \frac{ G_{ik} G_{kj}}{G_{kk}}.
\end{equation}
The identities for $i \neq k$
\begin{equation} \label{eqn:gijid}
G_{ik} = - G_{kk} \sumk_j h_{kj}G_{ji}^{(k)}.
\end{equation}
We have also the Ward identity
\begin{equation}
\sum_j |G_{ij}|^2 = \frac{1}{ \eta} \Im [ G_{ii} ] \label{eqn:ward}.
\end{equation}
All of the above identities hold with $G$ replaced by $\Gk$ or $\tilGk$. Lastly we have the resolvent identities, which connect $\Gk$ and $\tilGk$,  for $i, j \neq k$,
\begin{equation} \label{eqn:resolv}
\Gk_{ij} = \tilGk_{ij} + \sumk_l \Gk_{il} h_{kl} \tilGk_{lj} = \tilGk_{ij} + \sumk_l \tilGk_{il} h_{kl} \Gk_{lj} .
\end{equation}
\end{proposition}
By the Schur complement formula \eqref{eqn:schur}, we have
\beqn
G_{kk}(z)=\frac{1}{-\sumk_{i}h_{ik} -z- \mN (z)-{\cal E}_k},
\eeqn
where the error term is given by
\begin{align}\label{er}
{\cal E}_k(z):=\sumk_{i\neq j} h_{ki}G_{ij}^{(k)}(z)h_{jk}-\sumk_{i} (h_{ki}^2-\frac{1}{N})G_{ii}^{(k)}(z)-(\mN(z)-\mk(z)).
\end{align}

\begin{definition} \label{def:controlevents}
Let $c_*$ be the constant introduced in \eqref{eqn:immfc} below.  Note that $c_*$ depends only on the choice of $L$ in $\D_L$.  For any $z \in \D_L$ we define the events,
\begin{align}\label{def:Xi}
\Xik(z)=\{\max_{i, j \neq k} |\tilGk_{ij}(z)| \leq 2/c_*\}, \quad k=1,2,\cdots, N+1,
\end{align}
and the event
\begin{equation}
\Xi(z) := \bigcap_k \Xik(z) \cap \{ |\mN(z) - \mfc (z) | \leq ( \log (N) )^{-1} \}.
\end{equation}
\end{definition}
\remark 
We will often omit the parameter $z$ and write $\Xik$ and $\Xi$ if the context is clear. 

We will need the following regularity estimates for $\mfc$.
\begin{lemma}\label{l:reg}
For any fixed $L>0$ we have that the imaginary part of $\mfc$ is bounded above and away from zero,
\begin{align} \label{eqn:immfc}
c_* \leq \mathrm{Im} [\mfc(z)]\leq |\mfc(z)|\leq 1,
\end{align}
uniformly for $z\in \D_L^{(1)}$.  The constant $c_*$ depends only on $L$.  

 Moreover, $\mfc$ is Lipschitz
\begin{equation}\label{derbound}
| \mfc'(z) |\leq C 
\end{equation}
uniformly for all $ \Im[z] \geq 0$.
\end{lemma}
\begin{proof}
We have by the Cauchy-Schwarz inequality,
\begin{equation}\label{Cauchy}
|\mfc (z) |^2 \leq \int \frac{ \rho_G(x)\d x } { |x - z - \mfc(z) |^2} = \frac{ \Im [ \mfc (z) ] }{ \Im [ z + \mfc (z) ] } \leq 1,
\end{equation}
which gives us the upper bound in \eqref{eqn:immfc}. The lower bound follows from the fact that $\rfc$ is supported on the whole real axis.

To prove that $\mfc$ is Lipschitz, we differentiate both sides of the defining equation \eqref{e:mfc} to obtain 
\beqn
 \left( 1 - \int \frac{\rho_G(x) \d (x) }{ ( x - z - \mfc (z) )^2 } \right)\mfc' (z)  = \int \frac{\rho_G(x)\d x}{(x-z-\mfc(z))^2}.
\eeqn
From \eqref{Cauchy}, the right-hand side is bounded by $1$, and so we need only prove that the factor in front of $\mfc'(z)$ is bounded away from $0$. 
In the following we prove that
\begin{equation}\label{e:breal}
\mathrm{Re} \left[ \int \frac{ \d \mu_G (x) }{ ( x - z - \mfc (z) )^2 } \right]\leq 1 - c
\end{equation}
for a $c>0$ and \eqref{derbound} follows.

If the imaginary part of $z$, $\Im[z]\geq 2$, then $|x-z-\mfc(z)|\geq 2$, and \eqref{e:breal} is trivial. Therefore in the following we assume that the imaginary part of $z$ is bounded by $2$.

We first consider the regime $|\Re [z] | \geq L'$ for a fixed large $L' > 0$ chosen later.  Denote $z + \mfc (z) = a + b \i$.  WLOG take $\Re [z] \geq 0$ and so $a \geq L'-1$ because $| \mfc| \leq 1$.
Integration by parts yields,
\beq
\int \frac{ \d \mu_G (x) }{ ( x - z - \mfc (z) )^2 }
=\int \frac{-x \rho_G(x)\d x}{x-z-\mfc(z)}
=\int_{a-4}^{a+4} \frac{-x \rho_G(x)\d x}{x-z-\mfc(z)}
+\int_{[a-4,a+4]^c}\frac{-x \rho_G(x)\d x}{x-z-\mfc(z)}.
\eeq
Choosing $L'>0$ large enough we see that the second term is bounded by
\beqn
 \left|\int_{[a-4,a+4]^c}\frac{-x \rho_G(x)\d x}{x-z-\mfc(z)}\right|\leq \frac{1}{4}\int |x|\rho_G(x)\d x\leq \frac{1}{4}.
\eeqn
We rearrange the first term as 
\begin{align}\label{e:firstterm}
 \int_{a-4}^{a+4} \frac{-x \rho_G(x)\d x}{x-z-\mfc(z)}
 =-\int_{a-4}^{a+4}\rho_G(x)\d x-(z+\mfc(z)) \int_{a-4}^{a+4} \frac{ \rho_G(x)\d x}{x-z-\mfc(z)}.
\end{align}
For the first integral in \eqref{e:firstterm} we have the bound,
\begin{align}\label{firstt}
 \left|-\int_{a-4}^{a+4}\rho_G(x)\d x\right|
 \leq \frac{1}{\sqrt{2\pi}(a-4)}e^{-(a-4)^2/2}.
\end{align}
The second term in \eqref{e:firstterm} is somewhat harder due to the ``singularity'' at $x=a$. We write,
\begin{align}
 &(z+\mfc(z)) \int_{a-4}^{a+4} \frac{ \rho_G(x)\d x}{x-z-\mfc(z)}
 =(a+\i b)\left(\int_{0}^{4}\d y\frac{\rho_G(a+y)}{y-\i b}-\frac{\rho_G(a-y)}{y+\i b}\right)\notag\\
 =&(a+\i b)\left(\int_{0}^{4}\d y\frac{\rho_G(a+y)-\rho_G(a-y)}{y-\i b}\right)+(a+\i b) \int_{0}^4 \d y \frac{2\i b}{y^2+b^2}\rho_G(a-y).
\end{align}
The real part of above expression is bounded above by
\begin{align}\label{secondt}
 (a+3)\int_0^4 \d y \frac{\rho_G(a+y)-\rho_G(a-y)}{y}+\int_0^4 \d y \frac{2 b^2}{y^2+b^2}\rho_G(a-y)\leq \frac{2(a^2+3a+1)}{\sqrt{2\pi}(a-4)}e^{-(a-4)^2/2}.
\end{align}
Therefore \eqref{e:breal} follows after taking $L' >0$ sufficiently large.

We now consider the case $|\Re[z]|\leq L'$.  We have that the imaginary part of $\mfc$ is bounded below by $c'$. In this case we have
\begin{align}
\notag\mathrm{Re} \left[ \int \frac{ \rho_G(x)\d x }{ ( x - z - \mfc (z) )^2 } \right]
=&\int
\rho_G(x)\d x\frac{(\Re[x-z-\mfc])^2-(\Im[z+\mfc])^2 }{|x-z-\mfc|^4}\\
\notag\leq&
\int
\frac{\rho_G(x)\d x }{|x-z-\mfc|^2}-
(\Im[z+\mfc])^2 \int \frac{ \rho_G(x)\d x}{|x-z-\mfc|^4}\\
\leq& 1 - ( \Im [ \mfc (z) ] )^2\leq 1-(c')^2,
\end{align}
where in the second inequality we have used the Cauchy-Schwarz inequality
\begin{align}
 (\Im[z+\mfc])^2 \int \frac{ \rho_G(x)\d x}{|x-z-\mfc|^4}
 \geq (\Im[z+\mfc])^2 \left(\int \frac{ \rho_G(x)\d x}{|x-z-\mfc|^2}\right)^2=(\Im[\mfc])^2.
\end{align}
This finishes the proof of \eqref{e:breal}.
\end{proof}

The following proposition provides a discrete version of the estimate \eqref{e:breal}.
\begin{proposition}\label{p:stab}  There is a constant $c_1$ depending only on the constant $c_*$ from \eqref{eqn:immfc} so that the following holds uniformly for $z \in \DLo$. 
Suppose that $w_1,w_2,\cdots, w_{N+1}$ is a sequence of real numbers satisfying
\beqn
\frac{1}{N}\sum_{k}\frac{1}{w_k-z-\mfc(z)}=\mfc(z)+\epsilon(z)
\eeqn
where the error term satisfies 
\beq
|\epsilon(z)|\leq c_1
\eeq
Then we have
\begin{align} \label{eqn:stab}
\frac{c_*^2}{16}\leq \left| 1-\frac{1}{N}\sum_{k}\frac{1}{(w_k-z-\mfc(z))^2}\right|\leq 1 + \frac{1}{c_*^2}
\end{align}
In fact we can take $c_1 = \min \{ c_*^3/16, c_*/2 \}$.
\end{proposition}
\begin{proof} 
The upper bound of \eqref{eqn:stab} is an immediate consequence of the lower bound of \eqref{eqn:immfc}. 
For the lower bound of \eqref{eqn:stab} we compute
\begin{align}
\mathrm{Re}\left[\frac{1}{N}\sum_{k}\frac{1}{(w_k-z-\mfc(z))^2}\right]
=&\frac{1}{N}\sum_k\frac{(\mathrm{Re}[w_k-z-\mfc])^2-(\mathrm{Im}[z+\mfc])^2}{|w_k-z-\mfc|^4} \notag \\
\leq& \frac{1}{N}\sum_k\frac{1}{|w_k-z-\mfc|^2}-\frac{1}{N}\sum_k\frac{(\mathrm{Im}[z+\mfc])^2}{|w_k-z-\mfc|^4}\notag \\
\leq& \frac{\mathrm{Im}[\mfc+\epsilon]}{\mathrm{Im}[\mfc+\eta]}-\frac{1}{2}(\mathrm{Im}[\mfc+\epsilon])^2 \notag \\
\leq& 1 +\frac{|\epsilon|}{\Im[\mfc]}-\frac{(\Im[\mfc])^2}{8} 
\leq 1-\frac{c_*^2}{16}.
\end{align}
The last line follows from our choice of $c_1$.  In the third  inequality we used the Cauchy-Schwarz inequality
\begin{equation}
\frac{1}{N+1} \sum_k \frac{ (\Im [ z + \mfc ] )^2 }{ |w_k - z - \mfc  |^4 } \geq \left( \frac{1}{N+1} \sum_{k}\frac{ \Im [ z + \mfc ] }{ |w_k - z - \mfc  |^2} \right)^2 = (\Im [ \mfc  + \epsilon ] )^2.
\end{equation}
\end{proof}
We also need the following concentration of measure result.
\begin{proposition}\label{sum}
For any $z\in \D_L$, with overwhelming probability, we have
\begin{equation} \label{e:sum}
\left|\frac{1}{N}\sum_{k}\frac{1}{-\sumk_{i}h_{ik}-z-\mfc(z)}-\mfc(z)\right| =O\left(\frac{\xi}{q}\right)
\end{equation}
\end{proposition}

\begin{proof}
From the assumption \eqref{asup:moment}, we know that with overwhelming probability, we have 
\begin{align}\label{e:b_entry}
\max_{i<j}|h_{ij}|\leq \frac{\xi}{q \log N},\qquad \max_k\left|\sumk_i h_{ki}\right|\leq \xi.
\end{align} 
Thus we can do a truncation. For this, we define two cutoff functions.  First, take $\rho(x)=x$ on $[-\xi/(q\log N), \xi/(q\log N)]$, and $\rho(x)=\textrm{sign}(x)\xi/(q\log N)$ outside the interval.  Secondly, let $\chi(x)=x$ on $[-\xi, \xi]$, and $\chi(x)=\textrm{sign}(x)(\xi+1)$ for $|x| > \xi +1$, and with $\chi$ having its first three derivatives bounded. We also define the functions
\beqn
F(\{h_{ij}\}_{i<j})=\frac{1}{N}\sum_{k}\frac{1}{g_k},\qquad g_k:=g_k(\{h_{ik}\}_{i\neq k})=-\chi\left(\sumk_{i}\rho(h_{ik})\right)-z-\mfc(z).
\eeqn
We claim that $g_k$ is bounded away from $0$ on $\D_L$.  We have $|g_k| \geq | \Im [z] | \geq 2$ on $\DLz$.  On $\DLo$ we have $|g_k| \geq \Im [ \mfc ] \geq c_*$ and on $\DLt$ we have $|g_k| \geq \xi -2$. 
 Therefore, for any pair $i<j$, we have
\begin{align}
&\left|F(h_{12}, \cdots, h_{ij}', \cdots, h_{NN+1})-F(h_{12}, \cdots, h_{ij}, \cdots, h_{N N+1})\right|\notag \\
=&\frac{1}{N}\left|\frac{g_i(h_{ij})-g_i(h_{ij}')}{g_i(h_{ij})g_i(h_{ij}')}+\frac{g_j(h_{ij})-g_j(h_{ij}')}{g_j(h_{ij})g_j(h_{ij}')}\right|\leq \frac{C}{N}|\rho(h_{ij})-\rho(h_{ij}')|\leq \frac{2C\xi}{N q \log N}
\end{align}
Therefore by McDiarmid's inequality, we get
\begin{align}\label{e:Mac}
\PP\left(|F(\{h_{ij}\}_{i<j})-\EE[F(\{h_{ij}\}_{i<j})]|\geq \frac{t \xi}{q\log N}\right)\leq \e^{ - Ct^2}.
\end{align}
Hence, by taking $t=\log N$, we have that the following holds with overwhelming probability
\begin{align}\label{e:Mac_ineq}
|F(\{h_{ij}\}_{i<j})-\EE[F(\{h_{ij}\}_{i<j})]|\leq \xi/q.
\end{align}
Moreover, since $\max_{i<j}|h_{ij}|\leq \xi/(q \log N)$ with overwhelming probability we have
\begin{align}\label{e:hequ}
F(\{h_{ij}\}_{i<j})=\frac{1}{N} \sum_{k} \frac{1}{ - \chi\left(\sumk_i h_{ik} \right)- z - \mfc (z) } ,
\end{align}
By a similar argument as above we see that $|-\chi(\sumk_{i}h_{ik})-z-\mfc(z)|$ is bounded away from zero uniformly for $z \in \D_L$. Therefore both sides of \eqref{e:hequ} are uniformly bounded. By taking expectation on both sides of \eqref{e:hequ}, for any fixed $D>0$, we have 
\begin{align}\label{e:small_error}
\EE[ F(\{h_{ij}\}_{i<j})]= \EE\left[\frac{1}{N} \sum_{k} \frac{1}{ - \chi\left(\sumk_i h_{ik} \right)- z - \mfc (z) }\right]+O(N^{-D}),
\end{align}
given that $N$ is large enough.
Moreover, from our choice of $\chi$, we see that the first three derivatives of $(-\chi(\sumk_{i}h_{ik})-z-\mfc(z))^{-1}$ with respect to $h_{ik}$'s are bounded. By the standard Linderberg replacement trick one can show that 
\begin{align}\label{e:exp1}
\EE\left[\frac{1}{-\chi\left(\sumk_{i}h_{ik}\right)-z-\mfc(z)}\right]
\notag=&\EE\left[\frac{1}{-\chi\left(\sum_{i=1}^{N}x_i\right)-z-\mfc(z)}\right]+ O\left(\frac{1}{q}\right)\\
=&\EE\left[\frac{1}{-\sum_{i=1}^{N}x_i-z-\mfc(z)}\right]+ O\left(\frac{1}{q}\right)
\end{align}
where $x_{j}$'s are Gaussian random variables with mean zero and variance $1/N$.  The implicit constant in the error term depends only on the bound for third moment of the $h_{ij}$. Since $\sum_{i=1}^{N}x_i\stackrel{d}{=}{\cal N}(0,1)$ we have by definition,
\begin{align}\label{e:exp2}
\EE\left[\frac{1}{-\sum_{i=1}^{N}x_i-z-\mfc(z)}\right]=\mfc(z).
\end{align}
By combining  \eqref{e:small_error}, \eqref{e:exp1} and \eqref{e:exp2}, we get
\beqn
|\EE[F(\{h_{ij}\}_{i<j})]-\mfc|=O\left(\xi/q\right).
\eeqn
The claim follows from this and \eqref{e:Mac_ineq}.
\end{proof}

\subsection{Proof of Local Law}
In this subsection we prove Theorem \ref{t:lclaw}.  Recall the definition of the events $\Xik$ and $\Xi$ in Definition \ref{def:controlevents}, the Green's functions $G, \Gk, \tilGk$ and error term $\E_k$ in \eqref{er}.

We first prove that with overwhelming probability $\Xi(z)$ holds uniformly for $\eta \geq 2$. Then we will bootstrap the estimate by taking advantage of the self-consistent equation of 
the Stieltjes transform and the following error control, the proof of which is postponed to Section \ref{s:proof}.

\begin{proposition}\label{p:er}
For any $z\in \D_L$ all of the following holds with overwhelming probability.  We have,
\begin{align}\label{e:etry}
\1_{\Xik(z)}|\Gk_{ij}(z)-\tilGk_{ij}(z)|\leq \xi \Psi,
\end{align}
and
\begin{align}\label{e:entryav}
\1_{\Xik(z)}  \left|  \sumk_j h_{kj} \Gk_{ji} \right| \leq \xi \Psi
\end{align}
and
\begin{align}\label{e:er}
\1_{\Xik(z)}\E_k(z)\leq \xi \Psi.
\end{align}
where the control parameter $\Psi$ is introduced in \eqref{controlPsi}.
\end{proposition}

The proof of the local law is decomposed into two propositions.
\begin{proposition}\label{p:init}
With overwhelming probability, $\Xi(z)$ holds uniformly for $\eta \geq 2$.
\end{proposition}
\begin{proof}

Since $|G_{ij}|$, $|G_{ij}^{(k)}|$ and $|\tilde{G}^{(k)}_{ij}|$ are all bounded by $1/\Im [z]$, $\Xik(z)$ holds deterministically for any $z$ with $\eta \geq 2$. As for the error estimate of $m_N(z)$, we prove the following stronger result: for $\eta \geq 2$ we have,
\begin{align}\label{e:estm}
|\mN(z)-\mfc(z)|\leq \xi \Psi
\end{align}
with overwhelming probability.

For any $z$ with $\Im[z]\geq 2$, we denote the event ${\cal A}(z)$ such that \eqref{e:sum} and $\cap_{k=1}^{N+1}\{ |\E_k|\leq \xi \Psi\}$ hold. Then from Proposition \ref{sum} and \ref{p:er}, ${\cal A}(z)$ holds with overwhelming probability and so we can assume for the remainder of the proof that $\A(z)$ holds. 
Starting from the Schur complement formula \eqref{eqn:schur} we have,
\beqn
G_{kk}(z)=\frac{1}{-\sumk_{i}h_{ik} -z-\mN(z)-{\cal E}_k}
=\frac{1}{-\sumk_{i}h_{ik} -z-\mN(z)}+O(\xi\Psi ),
\eeqn
where we used the fact that $\Im [z] \geq 2$ to expand the denominator which is bounded away from $0$. After averaging over $k$ we get,
\begin{align}\label{mnerror}
\mN(z)=\frac{1}{N}\sum_{k}\frac{1}{-\sumk_{i}h_{ik} -z-\mN}+O(\xi\Psi )
\end{align}
We obtain the following by taking the difference of \eqref{mnerror} and \eqref{e:sum}
\begin{align}\label{e:inib}
\mN-\mfc=\frac{\mN-\mfc}{N}\sum_{k}\frac{1}{(-\sumk_{i}h_{ik} -z-\mN)(-\sumk_{i}h_{ik} -z-\mfc)}+O(\xi\Psi ).
\end{align}
The imaginary part of the denominators on the RHS of \eqref{e:inib} is bounded below by
\beqn
\Im [\sum_{j\neq k}h_{kj}+ z+\mN ] \geq 2, \qquad \Im [ \sum_{j\neq k}h_{kj} +z+\mfc ] \geq 2.
\eeqn
Therefore,
\begin{align}\label{e:stb}
\left|\frac{1}{N}\sum_{k}\frac{1}{(\sumk_{i}h_{ik} +z+\mN)(\sumk_{i}h_{ik} +z+\mfc)}\right|\leq \frac{1}{2}.
\end{align}
Combining \eqref{e:inib} and \eqref{e:stb} we obtain
\beqn
|\mN(z)-\mfc(z)|=O (\xi\Psi).
\eeqn
This finishes the proof.
\end{proof}

The following proposition is the key bootstrap principle.  The weak estimates of the event $\Xi (z)$ are used to prove stronger estimates on the Green's function and Stieltjes transform.
\begin{proposition}\label{p:better}  For $z \in \D_L$ the following holds with overwhelming probability on the event $\Xi(z)$. For the Stieltjes transform we have 
\begin{equation}
 | \mN (z) - \mfc (z) |\leq \xi \Psi  . \label{eqn:mbetter}
\end{equation}
Let $X_{ij}$ denote any of the resolvent entries 
$G_{ij}(z)$, $\Gk_{ij}(z)$ or  $\tilGk_{ij}(z)$.  We have,
\begin{equation}\label{eqn:Gbetter}
\left|  X_{ij}(z) - \frac{\delta_{ij}}{ - \sum^{(j)}_{l} h_{lj} - z - \mfc (z) } \right|\leq \xi \Psi .
\end{equation}
Recall  $\Psi$ is the control parameter introduced in \eqref{controlPsi}.
\end{proposition}

\begin{proof} 

For any $z\in \D_L$, we define the event ${\cal A}(z)$ such that \eqref{e:sum}, \eqref{e:b_entry} and $\Xi(z)$ hold, and the following holds:
\beqn \label{eqn:biga}
\bigcap_{k=1}^{N+1}\left(\{\E_k(z)\leq \xi \Psi\}\cap \{\max_{i,j\neq k}|\Gk_{ij}(z)-\tilGk_{ij}(z)|\leq \xi \Psi\}\cap \{\max_{j\neq k}|  \sumk_j h_{kj} \Gk_{ji} | \leq \xi \Psi \}\right).
\eeqn
By Proposition \ref{sum} and \ref{p:er}, and our assumption,  ${\cal A}(z)$ holds with overwhelming probability on $\Xi(z)$, and so we need only to prove that on $\A(z)$, \eqref{eqn:mbetter} and \eqref{eqn:Gbetter} hold.

We claim that on the event $\A(z)$, we have the uniform bound for sufficiently large $N$,
\beqn
\left| - \sumk_i h_{ik} -z - \mN(z)  \right| \geq \min\{2, \frac{c_*}{2}\} > 0,
\eeqn
where $c_*$ is the constant from \eqref{eqn:immfc}.
If $z \in \DLz$ then $\Im[z] \geq 2$.  If $z \in \DLo$ then $\Im [ \mfc (z) ] \geq c_*$ and so for sufficiently large $N$, $\Im [ \mN (z) ] \geq c_*/2$ by the definition of $\A(z)$.  Lastly, if $z \in \DLt$ then by the definition of $\A(z)$
\beqn
\left| - \sumk_i h_{ik} -z - \mN  \right|  \geq |z| - \left| \sumk_i h_{ik}\right| - |\mN| \geq \xi -2\geq 2.
\eeqn

Therefore we may expand the Schur complement formula \eqref{eqn:schur} and obtain,
\begin{align}
\label{eqn:mN}\mN(z)=\frac{1}{N}\sum_{k}\frac{1}{-\sumk_{i}h_{ik} -z-\mN}+O (\xi \Psi ).
\end{align}
Again, by the definition of $\A(z)$ we have
\begin{align}
\label{eqn:mfc}\mfc(z)=\frac{1}{N}\sum_{k}\frac{1}{-\sumk_{i}h_{ik} -z-\mfc}+O(\xi\Psi ).
\end{align}
Taking the difference of \eqref{eqn:mN} and \eqref{eqn:mfc} we get 
\begin{align}\label{eqn:difm}
\left(1-\frac{1}{N}\sum_{k}\frac{1}{(-\sumk_{i}h_{ik} -z-\mfc)^2} -{\cal E}\right)(\mN-\mfc)=O(\xi \Psi)
\end{align}
where the error term is given by
\beqn
|{\cal E}|=\left|\frac{1}{N}\sum_{k}\frac{\mN-\mfc}{(-\sumk_{i}h_{ik} -z-\mN)(-\sumk_{i}h_{ik} -z-\mfc)^2}\right|\leq \frac{C}{\log N},
\eeqn
We used again that the denominator is bounded away from zero on $\D_L$. By Proposition \ref{p:stab} and the assumption that  \eqref{e:sum} holds on $\A$, we see that the prefactor in front of $(\mN - \mfc)$ in \eqref{eqn:difm} is bounded away from zero by a constant on $\DLo$. This is also true on $\DLz$ due to the fact that $\Im[z] \geq 2$, and is true on $\DLt$ by the fact that $|-\sumk_{i}h_{ik} -z-\mfc| \geq \xi -2$ on $\A(z)$. We conclude \eqref{eqn:mbetter}.

On the event $\A$ we have for the diagonal resolvent terms,
\begin{align}
G_{kk} &= \frac{1}{ - \sumk_i h_{ik} - z - \mfc (z) + ( \mN (z) - \mfc (z) ) + \E_k } \notag \\
&= \frac{1}{ - \sumk_i h_{ik} - z - \mfc (z) } + O(\xi \Psi ),
\end{align}
where we used \eqref{eqn:mbetter}.  For the off-diagonal resolvent entries we apply the identity \eqref{eqn:gijid} and obtain
\begin{align}
|G_{ik}| = |G_{kk} | \left| \sumk_j h_{kj} \Gk_{ji} \right| =O(\xi\Psi),
\end{align}
using the fact that $|G_{kk}|\leq 2/c_*$ and $| \sumk_j h_{kj} \Gk_{ji} |\leq \xi \Psi$ on $\A$ (recall \eqref{eqn:biga}).  For the entries of the $k$th minor we have by \eqref{eqn:minorid} and \eqref{eqn:gijid}
\beqn
|G_{ij} - \Gk_{ij} | = \left| \frac{G_{ik} G_{kj} }{ G_{kk} } \right| =\left| \sumk_j h_{kj} \Gk_{ji} \right| |G_{kj}| =O(\xi \Psi).
\eeqn
The estimates for $\tilGk_{ij}$ then follow from the definition \eqref{eqn:biga} of $\A$  on which we have $|\Gk_{ij}(z)-\tilGk_{ij}(z)|\leq \xi \Psi$.
\end{proof}

\noindent \emph{Proof of Theorem \ref{t:lclaw}.}  Once Propositions \ref{p:init} and \ref{p:better} have been established, the local law follows from a standard induction scheme as given in \cite{EKYYgeneral}.  We provide here a sketch for the sake of completeness.  Note that on $\D_L$ all the resolvent entries are Lipschitz in the parameter $z$ with constant less than $CN^2$.  It therefore suffices to work on the lattice
\beqn
\Lambda = \D_L \cap N^{-4} ( \zz + \i \zz )
\eeqn
which has cardinality $|\Lambda | \leq 8LN^{8}\xi$.  Fix now a sequence of sets of points $L_k = \{E + (2 -  jN^{-4})\i \in \Lambda: j\leq k\}$, and let $\A_k$ be the event that the conclusion of Theorem \ref{t:lclaw} holds uniformly for points in $L_k$. We know by Proposition \ref{p:init} that $\A_0$ holds with overwhelming probability. By the Lipschitz continuity of the resolvent entries, for any $z\in L_{k+1}$, the event $\Xi (z)$ holds on $\A_k$.  Hence, by Proposition \ref{p:better}, the event $\A_{k+1}$ holds with overwhelming probability on $\A_k$.  Hence fixing a large number $D>0$ we see by induction that 
\beqn
\pp [ \A_k^c] \leq \pp [ \A_{k-1}^c] + \pp [ \A_{k-1} \cap \A_k^c] \leq k N^{-D} + N^{-D} = (k+1)N^{-D}.
\eeqn
Hence,
\beqn
\pp [ \cap_{k=0}^M A_k ] \geq 1 - \sum_{k=0}^{N^4-1} (k+1) N^{-D} \geq 1 - N^{-D+8}.
\eeqn
 This yields the claim. \qed

\begin{proof}[Proof of Corollary \ref{cor:rig}]  
We want to apply the arguments of Section 7 of \cite{DBM} to conclude rigidity.  However, since the norm of $H$ is not bounded we need a slight modification of the argument.
Using the local law \eqref{e:locallaw}, we apply \cite[Lemma 7.17]{DBM} with $E_1 = -2\xi$ and $E_2 \in (-L/2, L/2)$ and obtain
\begin{equation}\label{e:bound1}
\left|\int_{E_1}^{E_2}  \rhofc (E)\d E - \int_{E_1}^{E_2} \rho_N (E)\d E \right| \leq C\xi^2q^{-1},
\end{equation}
with overwhelming probability, where $\rho_N$ the empirical eigenvalue distribution of $H$. 
We also have that $\rho_N$ is supported in $[-2\xi,2\xi]$ with overwhelming probability.  By letting $\Im[z]$ goes to zero on both sides of \eqref{e:mfc}, we get
\begin{align}
\sqrt{2\pi} &= \int_{\RR} \frac{ e^{-x^2/2}\d x }{ ( x - E- \mathrm{Re} [ \mfc (E) ] )^2 + \pi ^2\rhofc (E)^2 } \notag \\
&\leq \int_{|x-E|\leq |E|/2} \frac{ e^{-x^2/2}\d x }{ \pi ^2\rhofc (E)^2 }  + \int_{ |x - E| > |E|/2} \frac{e^{-x^2/2}\d x } { ( x - E - \mathrm{Re} [ \mfc (E) ] )^2 } \notag \\
& \leq  \frac{2e^{-E^2/8}}{\pi^2|E|\rhofc(E)^2}+ \frac{4}{ (|E|-2)^2}
\end{align}
where we used that $|\mfc|\leq 1$ from \eqref{eqn:immfc}. We therefore have that $\rhofc(E)\leq e^{-E^2/8}$ for sufficiently large $|E|$, and therefore for any $E \leq E_1$.  Therefore,
\begin{equation}\label{e:bound2}
\int_{-\infty}^{E_1} \rhofc (E)\d E \leq N^{-D}
\end{equation}
for any $D>0$.  Combining \eqref{e:bound1} and \eqref{e:bound2} we get
\beqn
\left |\int_{-\infty}^{E_2} \rhofc (E)\d E- \int_{-\infty}^{E_2}  \rho_N(E) \d E\right| \leq C\xi^2q^{-1}
\eeqn
By a standard argument (see, e.g., \cite[Theorem 2.2]{EYYrig}), we obtain the following rigidity estimate for the bulk eigenvalues.  For any index $i$ such that $\gamma_i\in[-L/2, L/2]$, we have
\beqn
 |\lambda_i(H)-\gamma_i|\leq C\xi^2q^{-1}.
\eeqn
\end{proof}
The proofs of Corollary \ref{cor:del} from Theorem \ref{t:lclaw} is standard and we omit the details. This argument can be found, for example, in \cite[Theorem 2.16]{EKYYSparse2}.

\subsection{Proof of Proposition \ref{p:er}}\label{s:proof}
We show that Proposition \ref{p:er} follows from the following large deviation estimates.

\begin{proposition}\label{p:ld1} 
 For any $z\in \D_L$, the following holds with overwhelming probability: for $l \geq 1$ and $j \neq k$ we have,
\begin{equation} \label{eqn:hg}
\1_{\Xik(z)} \sumk_{i_1, ... i_l} h_{k i_1 } \tilGk_{i_1, i_2} h_{k i_2 } \cdots h_{k i_l } \tilGk_{i_l j } \leq \xi  \left( \frac{1}{q} + \frac{1}{ ( N \eta)^{1/4} } \right)^{l},
\end{equation}
and for $l\geq 1$ and $i, j \neq k$
\begin{equation} \label{eqn:ghg}
\1_{\Xik(z)} \sumk_{i_1, ..., i_l} \tilGk_{i i_1} h_{k i_1} \tilGk_{i_1 i_2} h_{k i_2} \cdots h_{k i_l } \tilGk_{i_l j } \leq \xi   \left( \frac{1}{q} + \frac{1}{ ( N \eta)^{1/4} } \right)^{l}.
\end{equation}
For $l \geq 3$ we have
\begin{equation} \label{eqn:hgh}
\1_{\Xik(z)} \sumk_{i_1, ..., i_l} h_{k i_1} \tilGk_{i_1 i_2} h_{k i_2} \cdots \tilGk_{i_{l-1} i_l } h_{k i_l } \leq \xi   \left( \frac{1}{q} + \frac{1}{ ( N \eta)^{1/4} } \right)^{l-2}.
\end{equation}
\end{proposition}
We have better estimates for smaller $l$.
\begin{proposition}\label{p:ld2} 
 For any $z\in \D_L$, the following holds with overwhelming probability: for $j \neq k$ we have,
\begin{equation} \label{eqn:hgl1}
\1_{\Xik(z)} \sumk_{i} h_{k i } \tilGk_{i, j}  \leq \xi  \left( \frac{1}{q} + \frac{1}{ ( N \eta)^{1/2} } \right),
\end{equation}
and for $i, j \neq k$
\begin{equation} \label{eqn:ghgl1}
\1_{\Xik(z)} \sumk_{i_1} \tilGk_{i i_1} h_{k i_1} \tilGk_{i_1 j}  \leq \xi   \left( \frac{1}{q} + \frac{1}{ ( N \eta)^{1/2} } \right).
\end{equation}
We have also
\begin{equation} \label{eqn:hghl2}
\1_{\Xik(z)} \sumk_{i,j} \tilGk_{ij}\left(h_{k i} h_{k j} -\delta_{ij}N^{-1}\right) \leq \xi   \left( \frac{1}{q} + \frac{1}{ ( N \eta)^{1/2} } \right),
\end{equation}
and
\begin{equation} \label{eqn:hghl3}
\1_{\Xik(z)} \sumk_{i_1, i_2, i_3} h_{k i_1} \tilGk_{i_1 i_2} h_{k i_2} \tilGk_{i_{2} i_3 } h_{k i_3 } \leq \xi   \left( \frac{1}{q} + \frac{1}{ ( N \eta)^{1/2} } \right).
\end{equation}
\end{proposition}

\begin{proof}[Proof of Proposition \ref{p:er}.]In this proof we assume the event $\Xik(z)$ holds and omit from the notation. We define the following three quantities corresponding to \eqref{eqn:hg}, \eqref{eqn:ghg} and \eqref{eqn:hgh},
\begin{align}
\begin{split}
A_l(i,j):=& \sumk_{i_1, ... ,i_l } \tilGk_{i i_1} h_{k i_1} \tilGk_{i_1 i_2} ... h_{k i_l} \tilGk_{i_l j},\\
A_l(j):=&\sumk_{i_1, ... ,i_l }  h_{k i_1} \tilGk_{i_1 i_2}h_{k i_2} ... h_{k i_l} \tilGk_{i_l j},  \\
A_l:=&\sumk_{i_1, ... ,i_l }  h_{k i_1} \tilGk_{i_1 i_2}h_{k i_2} ... \tilGk_{i_{l-1}i_l}h_{k i_l} ,
\end{split}
\end{align}
for any $l\geq 1$.  In order to prove \eqref{e:etry} we iterate \eqref{eqn:resolv} and obtain
\begin{equation} \label{eqn:ben1}
\Gk_{ij} - \tilGk_{ij} = \sum_{l=1}^m A_l (i, j) + \sumk_n A_m (i, n) h_{k n} \Gk_{n j}.
\end{equation}
By \eqref{eqn:ghgl1} and \eqref{eqn:ghg} we have with overwhelming probability,
\beqn
A_1(i,j) \leq \xi \left(\frac{1}{q} + \frac{1}{ ( N \eta)^{1/2} }\right), 
\eeqn
and for $l\geq 2$, 
\beqn
A_l(i,j) \leq  \xi\left( \frac{1}{q} + \frac{1}{ (N \eta)^{1/4} } \right)^{l}.
\eeqn
In order to control the last term on the right-hand side of \eqref{eqn:ben1} we use the trivial bound $|\Gk_{nj}| \leq 1/\Im[z]\leq N$ on $\D_L$ and take $m$ large (but independent of $N$, say $m\geq 4(1/\nu+1)$, where $\nu$ is as in \eqref{e:defineXi}) and obtain
\beqn
|\Gk_{ij} - \tilGk_{ij}| \leq  \sum_{l=1}^m |A_l (i, j)|+ \frac{1}{N}\leq C\xi \left(\frac{1}{q} + \frac{1}{ ( N \eta)^{1/2} }\right)
\eeqn
with overwhelming probability.

In order to prove \eqref{e:entryav} we iterate \eqref{eqn:resolv} and obtain
\beqn
\sumk_i h_{ki} \Gk_{ij} =  A_1(j) + \sum_{l=2}^m A_l ( j) + \sumk_{n}A_{m}(n)h_{kn}\Gk_{nj}.
\eeqn
The first term $A_1(j)$ is controlled by \eqref{eqn:hgl1} and the terms in the summation are controlled by \eqref{eqn:ghg}.  For the last term we again use the trivial bound $|\Gk_{ij}| \leq N$ and take $m$ large.  This yields \eqref{e:entryav}. 

For \eqref{e:er} we write
\begin{align}
 \E_k=&\sumk_{ij} \Gk_{ij} ( h_{ki} h_{kj} - \delta_{ij} N^{-1} ) -(m_N-\mk)
 =  \sumk_{ij} \tilGk_{ij} ( h_{ki} h_{kj} - \delta_{ij} N^{-1} )\notag\\ 
 +& \sum_{l=3}^m A_{l}+ \sumk_{i_m,j}A_{m-1}(i_m)h_{ki_m}\Gk_{i_m j}h_{kj}- ( \mk - \tilmk ) - (\mN-\mk ).
\end{align}
The claim follows from \eqref{eqn:hgh}, \eqref{eqn:hghl2}, \eqref{eqn:hghl3}, \eqref{e:etry}, and the deterministic bound \cite[Lemma 6.5]{BGMCLTHeavyTail}
\beqn
|\mN(z)-\mk(z)|\leq (N\eta)^{-1}.
\eeqn
\end{proof}

For the proofs of Propositions \ref{p:ld1} and \ref{p:ld2} we introduce a graphical notation to encode terms appearing in our high moment estimates.  This is inspired by the methodology of \cite{iso2014}.  We require the following combinatorial definition and lemma.

\begin{definition} \label{def:graph}
We define a labeled colored undirected graph $\G$ to be a graph on $m$ black vertices $\{1,2,\cdots, m\}$ and on either $0$, $1$ or $2$ white vertices, labeled as $m+1$ and $m+2$ if they exist. Furthermore we demand that each white vertex is associated to a specific  index which we will denote by $i, j\in \{1,2,\cdots, N+1\}$. $\G$ may have multiple edges and self loops.  In this section we will refer to such labeled colored undirected graphs simply as graphs.
\end{definition}

We will use the graphs defined above to encode various summations of monomials of Green's function elements.  Black vertices will encode an index to be summed over, while white vertices will encode fixed indices that are not summed over.

Let $\G$ be a graph as defined in Definition \ref{def:graph}.  We associate to $\G$ and an index $k \in [[1, N+1]]$ the following sum,
\beqn
\V (\G, k) := \frac{1}{N^m} \sumk_{i_1, ..., i_m} \prod_{(a, b) \in \G} |\tilGk_{i_a i_b}(z) | .
\eeqn
When $a$ or $b$ equal the white vertex then $i_a$ and $i_b$ denote the associated index, as defined in Definition \ref{def:graph}. Note that since $|\tilGk_{ij}| = | \tilGk_{ji}|$ the above sum is well-defined even though $\G$ is undirected.

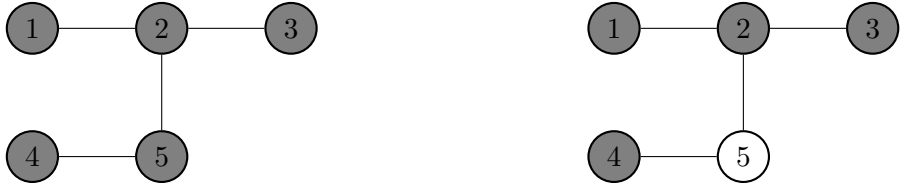
\begin{figure}[H]
\centering
\begin{minipage}[c]{0.5\textwidth}
\begin{center}
\begin{tikzpicture}[
blacknode/.style ={circle, draw=black, fill=black!50, thick},
whitenode/.style ={circle, draw=black, fill=white, thick},
]

\node[blacknode] (one) {1};
\node[blacknode] (two) [right=of one]{2};
\node[blacknode] (three) [right=of two] {3};
\node[blacknode] (four) [below =of one] {4};
\node[blacknode] (five) [right=of four] {5};

\draw[-] (one.east) -- (two.west);
\draw[-] (two.east) -- (three.west);
\draw[-] (four.east) -- (five.west);
\draw[-] (five.north) -- (two.south);

\end{tikzpicture}
\end{center}
\end{minipage}%
\begin{minipage}[c]{0.4\textwidth}
\begin{center}
\begin{tikzpicture}[
blacknode/.style ={circle, draw=black, fill=black!50, thick},
whitenode/.style ={circle, draw=black, fill=white, thick},
]

\node[blacknode] (one) {1};
\node[blacknode] (two) [right=of one]{2};
\node[blacknode] (three) [right=of two] {3};
\node[blacknode] (four) [below =of one] {4};
\node[whitenode] (five) [right=of four] {5};

\draw[-] (one.east) -- (two.west);
\draw[-] (two.east) -- (three.west);
\draw[-] (four.east) -- (five.west);
\draw[-] (five.north) -- (two.south);

\end{tikzpicture}
\end{center}
\end{minipage}
\caption{The graph on the left corresponds to the sum $ N^{-5} \sum^{(k)}_{i_1,\cdots, i_5} \tilde G_{i_1 i_2}^{(k)}\tilde G_{i_2 i_3}^{(k)} \tilde G_{i_2 i_5}^{(k)}\tilde G_{i_4 i_5}^{(k)}$.   If the white vertex appearing in the graph on the right is associated with the index $j$, then the graph corresponds to the sum 
$  N^{-4} \sum^{(k)}_{i_1,\cdots, i_4} \tilde G_{i_1 i_2}^{(k)}\tilde G_{i_2 i_3}^{(k)} \tilde G_{i_2 j}^{(k)}\tilde G_{i_4 j}^{(k)}$.}
\end{figure}

\begin{lemma}\label{L:graph_ineq}
Let $\G$ be a  connected graph with $m$ black vertices and $E$ edges.  If $\G$ has no white vertices then
\begin{equation}
\1_{\Xik} \V (\G, k) \leq C^{E} (N\eta)^{-(m-1)/2}. \label{eqn:nowhite}
\end{equation}
If $\G$ has at least one white vertex,  then
\begin{equation}
\1_{\Xik} \V (\G, k) \leq C^{E} (N\eta)^{-m/2}. \label{eqn:white}
\end{equation}
Above $C=\max\{2/c_*,1\}$, where $c_*$ is the constant in \eqref{def:Xi}.
\end{lemma}
\begin{proof}
In the following calculations we assume that $\Xik$ holds and we omit it from the notation.  The reader should therefore keep in mind that $|\tilGk_{ij}(z)|\leq C$.

First suppose that $\G$ has no white vertices. We can assume that $m$ is at least $2$ or the statement is trivial.  Let $\T$ be a spanning tree of $\G$.  By a relabeling of the vertices, we can assume  that the vertex $m$ is a leaf of $\T$ and that the vertex $m-1$ is adjacent to vertex $m$. We will use the Cauchy-Schwarz inequality and the Ward identity \eqref{eqn:ward} to sum over the summation index associated with the black vertex $m$. Let $\T_1$ be the tree obtained by removing the vertex $m$ and the edge $(m-1, m)$ from $\T$.  We have,
\begin{align}
\V (\G, k ) &\leq C^{E-m+1} \V ( \T, k) 
= C^{E-m+1} \frac{1}{N^m} \sumk_{i_1, ..., i_m} \prod_{ (a, b) \in \T } | \tilGk_{i_a i_b } | \notag \\
&=C^{E-m+1} \frac{1}{N^{m-1}} \sumk_{i_1, ..., i_{m-1}} \prod_{ (a, b) \in \T_1 } | \tilGk_{i_a i_b } | \frac{1}{N} \sumk_{i_m} |\tilGk_{i_{m-1} i_m } | \notag \\
&\leq C^{E-m+1} \frac{1}{N^{m-1}} \sumk_{i_1, ..., i_{m-1}} \prod_{ (a, b) \in \T_1 } | \tilGk_{i_a i_b } | \left( \frac{1}{N} \sumk_{i_m} |\tilGk_{i_{m-1} i_m } |^2 \right)^{1/2} \notag \\
&\leq C^{E-m+2} \frac{1}{ ( N \eta )^{1/2} } \V (\T_1, k).
\end{align}
The second to last line follows from Cauchy-Schwarz and in the last line we have used the Ward identity,
\beqn
  \frac{1}{N} \sumk_{i_m} |\tilGk_{i_{m-1} i_m } |^2 =\frac{\Im [\tilGk_{i_{m-1}{i_{m-1}}}]}{N\eta}\leq \frac{C}{N\eta}. 
\eeqn

We can now continue this procedure with tree $\T_1$ which consists of $m-1$ black vertices $\{1,2,\cdots m-1\}$.  That is, we choose a leaf of the tree $\T_1$ and repeat the above procedure of summing over the summation index associated with the leaf, applying Cauchy-Schwarz and the Ward identity.  To be more precise, we can assume that the vertex $m-1$ is a leaf attached to the vertex $m-2$.  If we denote by $\T_2$ the subtree obtained by removing the vertex $m-1$ and edge $(m-2, m-1)$ we obtain,
\begin{align}
\V ( \T_1, k) &= \frac{1}{N^{m-1} } \sumk_{i_1, \cdots i_{m-1} } \prod_{(a, b) \in \T_1 } | \tilGk_{i_a i_b } | \notag \\
&= \frac{1}{N^{m-2} } \sumk_{i_1, \cdots i_{m-2} } \prod_{ (a, b) \in \T_2 } | \tilGk_{i_a i_b } | \frac{1}{N} \sumk_{i_{m-1} } | \tilGk_{i_{m-2} i_{m-1} } | \notag \\
&\leq  \frac{1}{N^{m-2} } \sumk_{i_1, \cdots i_{m-2} } \prod_{ (a, b) \in \T_2 } | \tilGk_{i_a i_b } |  \left( \frac{1}{N} \sumk_{i_{m-1} } | \tilGk_{i_{m-2} i_{m-1} } |^2 \right)^{1/2} \notag \\
& \leq C \frac{1}{ ( N \eta )^{1/2} } \V ( \T_2, k).
\end{align}
 This yields another factor of $C (N \eta)^{-1/2}$.  We then inductively repeat this procedure on the subtree $\T_2$ consisting of $m-2$ back vertices, and so on. Each summation removes a leaf vertex and the edge adjacent to it, and results in a factor $C(N\eta)^{-1/2}$. 

We get a sequence of subtrees $\T_1, \T_2,\cdots, \T_{m-1}$, where the $n$th tree $\T_n$ is obtained after the $n$th summation.  The last tree $\T_{m-1}$ consists only of the root vertex and clearly $\V ( \T_{m-1}, k)= 1$ by definition. We see that
\begin{align}
\V (\G, k )  &\leq C^{E-m+2} \frac{1}{ ( N \eta )^{1/2} } \V (\T_1, k) \leq  C^{E-m+3} \frac{1}{ ( N \eta) } \V ( \T_2, k) \leq \cdots   \notag \\
& \leq   \left(C(N\eta)^{-1/2}\right)^{m-1} \V (\T_{m-1}, k) \leq C^E (N\eta)^{-(m-1)/2}.
\end{align}
This finishes the proof of \eqref{eqn:nowhite}. Suppose that $\G$ has a single white vertex. We repeat the same procedure as above but choose the spanning tree to be rooted at the white vertex.  After summing over $m$ black vertices we will have
\beqn
\V ( \G, k) \leq C^E( N \eta)^{-m/2} \V ( \T_{m}, k)
\eeqn
where $\T_{m}$ consists of only the white vertex.  Clearly $\V(\T_m, k) = 1$ by definition and we get the second inequality \eqref{eqn:white} in this case.

Now we can assume that $\G$ has two white vertices. Instead of taking out a spanning tree we choose a spanning forest consisting of two disjoint trees $\T\cup \T'$ each of which is rooted at a white vertex.  We then have,
\beqn
\V(\G,k)\leq C^{E-m}\V(\T,k)\V(\T',k)\leq C^E (N\eta)^{-m/2},
\eeqn
where the second inequality follows from the argument in the case that $\G$ has a single white vertex.  This finishes the proof. 
\end{proof}

We will lastly need the following construction for forming new graphs by grouping together black vertices. 
\begin{definition} Let $\G$ be a colored graph with $m$ black vertices labeled $\{1, 2, ..., m\}$ and $0, 1$ or $2$ white vertices. Let $\P$ be a partition on $\{1, 2, ..., m\}$. Note that $\P$ induces a partition on the black vertices of $\G$.  The graph $\G ( \P )$ is formed as follows. 

Firstly, $\G(\P)$ has the same number of white vertices as $\G$ and the same associated index.  To each block of $\P$ we create one black vertex for $\G(\P)$.  We now define the edges of $\G ( \P)$.  We first construct a map $\phi : V ( \G) \to V (\G(\P))$ from the vertices of $\G$ to the vertices of $\G(\P)$.  The white vertices of $\G$ are just sent to the corresponding white vertices of $\G ( \P)$.  A black vertex of $\G$ is associated with a unique block of $\P$ - the block that its label lies in - and we map this vertex to the black vertex of $\G ( \P)$ corresponding to this block.  Then, for each edge $(a, b)$ of $\G$ we construct an edge $(\phi(a), \phi(b))$ of $\G(\P)$.
\end{definition}

While lengthy, the above definition is very simple.  The following picture illustrates the process:
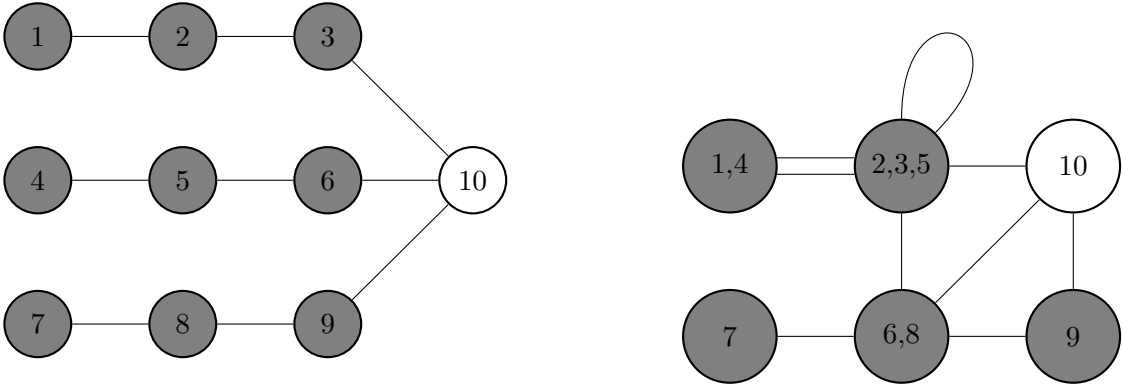
\begin{figure}[H]
 \centering
\begin{minipage}[c]{0.5\textwidth}
\begin{center}
\begin{tikzpicture}[
whitenode/.style ={circle, draw=black, fill=white, thick,minimum width =25pt},
blacknode/.style={circle, draw=black, fill=black!50, thick, minimum width=25pt},
]

\node[blacknode] (one) {1};
\node[blacknode] (two) [right=of one]{2}; 
\node[blacknode] (three) [right=of two] {3};
\node[blacknode] (four) [below= of one] {4};
\node[blacknode] (five) [right=of four] {5};
\node[blacknode] (six) [right= of five] {6};
\node[blacknode] (seven) [below=of four] {7};
\node[blacknode] (eight) [right= of seven] {8};
\node[blacknode] (nine) [right = of eight] {9};
\node[whitenode] (ten) [right=of six] {10};

\draw[-] (one.east) -- (two.west);
\draw[-] (two.east) --(three.west);
\draw[-] (four.east) --(five.west);
\draw[-] (five.east) -- (six.west);
\draw[-] (seven.east) -- (eight.west);
\draw[-] (eight.east) -- (nine.west);
\draw[-] (three.south east) --(ten.north west);
\draw[-] (six.east) -- (ten.west);
\draw[-] (nine.north east) -- (ten.south west);
 
\end{tikzpicture}
\end{center}
\end{minipage}%
\begin{minipage}[c]{0.5\textwidth}
\begin{center}
\begin{tikzpicture}[
whitenode/.style ={circle, draw=black, fill=white, thick, minimum width =35pt},
blacknode/.style={circle, draw=black, fill=black!50, thick, minimum width= 35pt},
]

\node[blacknode] (one) {1,4};
\node[blacknode] (two) [right=of one] {2,3,5}; 
\node[blacknode] (six) [below= of two] {6,8};
\node[blacknode] (seven) [left=of six] {7};
\node[blacknode] (nine) [right = of six] {9};
\node[whitenode] (ten) [right=of two] {10};

\draw[-] (one.10) -- (two.170);
\draw[-] (one.-10) -- (two.190);
\draw[-] (two) to[out=north east, in=north, looseness=10] (two);

\draw[-] (two.east) -- (ten.west);
\draw[-] (seven.east) -- (six.west);
\draw[-] (two.south) -- (six.north);
\draw[-] (six.east) -- (nine.west);
\draw[-] (nine.north) -- (ten.south);
\draw[-] (six.north east) -- (ten.south west);

\end{tikzpicture}
\end{center}
\end{minipage}
\caption{The graph on the left is the original graph $\cal G$ with $9$ black vertices and $1$ white vertices. The graph on the right is the new graph ${\cal G}(\cal P)$ after grouping together black vertices with respect to the partition ${\cal P}=\{\{1,4\}, \{2,3,5\}, \{6,8\}, 7,9\} \}$. }

\end{figure}

\begin{proof}[ Proof of Proposition \ref{p:ld1}.] In the following proof we work on the event $\Xik(z)$ and we omit it from the notation.  The reader should therefore keep in mind that $|\tilGk_{ij}(z)|\leq C$.  We will compute high order moments of the terms in all three expressions \eqref{eqn:hg}, \eqref{eqn:ghg} and \eqref{eqn:hgh}.  The claim then follows by a high-moment Chebyshev inequality.  

We begin with \eqref{eqn:hg}.  Its $2p$-th moment is
\begin{align}
\ee \left| \sumk_{i_1\cdots i_l} h_{k i_1 } \tilGk_{i_1, i_2} h_{k i_2 } \cdots h_{k i_l } \tilGk_{i_l j } \right|^{2p} &= \sumk_{i_1 \cdots i_{2pl}} \ee [ h_{k i_1} X_{i_1 i_2} \cdots X_{i_l j} h_{k i_{l+1} }X_{i_{l+1} i_{l+2} }
\cdots X_{i_{2lp} j} ], \label{eqn:unpart1}
\end{align}
where $X_{ij}$ denotes either $\tilGk_{ij}$ or its complex conjugate $(\tilGk_{ij})^*$. For us the distinction will not be important as we will be bounding the terms in the above expression by their absolute value.
Since $\tilGk$ is independent of the $k$-th row of $H$ we can first take the expectation over this row.  We then re-arrange the above expectation into all possible partitions of the summation indices $\{1, 2, \cdots, 2pl\}$.  Let $\P$ be such a partition.  Note that due to the fact that the $h_{ij}$ are independent centered random variables, the size of each block of $\P$ is at least $2$, otherwise the expectation will vanish. Therefore, \eqref{eqn:unpart1} can be bounded by
\begin{align}\label{eqn:unpart_bound}
&\ee \left| \sumk_{i_1\cdots i_l} h_{k i_1 } \tilGk_{i_1, i_2} h_{k i_2 } \cdots h_{k i_l } \tilGk_{i_l j } \right|^{2p} \notag \\
\leq &\sum_{\P}  \frac{C_{\P}}{ q^{2lp-2|\P|}}\left(\frac{1}{N^{|\P|}}\sumk_{\{i_1,i_2\cdots i_{2pl}\}\sim\P}\ee \left[ \left| X_{i_1 i_2} \cdots X_{i_l j} X_{i_{l+1} i_{l+2}}\cdots X_{i_{2lp}j} \right| \right] \right).
\end{align}
Above, the first sum is over all possible partitions of the set $\{1, 2, ..., 2pl\}$ such that the size of each block of $\P$ is at least $2$, and $|\P|$ denotes the number of blocks in $\P$. The constant $C_{\P}$ depends only on the partition $\P$. The second sum is over all the possible assignments of the indices $\{i_1,i_2\cdots i_{2pl}\}$ which induces the same partition as $\P$ of the set $\{1, 2, ..., 2pl\}$ (we recall that $\{i_1,i_2\cdots i_{2pl}\}$ induces a partition of the set $\{1, 2, ..., 2pl\}$, such that $a$ and $b$ are in the same block if and only if $i_a=i_b$).

We first construct the graph $\G$ with black vertices labelled by $1, 2, ..., 2lp$ and one white vertex associated with the index $j$. We create an edges between the vertices $(i, i+1)$ for $i=nl+1, nl+2, ..., (n+1)l-1$ and $n=0, 1,2, ...,2p-1$, and an edge between the black vertices $nl$ and the white vertex $j$, for $n=1, ..., 2p$.  The graph $\G$ represents the unpartitioned sum \eqref{eqn:unpart1}.  

To each partition $\P$ we then get a graph $\G = \G (\P)$ as defined above. By Lemma \ref{L:graph_ineq}, we see
\beqn
\frac{1}{N^{|\P|}}\sumk_{\{i_1,i_2\cdots i_{2pl}\}\sim\P}\ee [ \left| X_{i_1 i_2} \cdots X_{i_l j} X_{i_{l+1} i_{l+2}}\cdots X_{i_{2lp}j} \right| ]= \ee [\V ( \G ( \P ) , k )] \leq C^{2lp}(N\eta)^{-|\P|/2},
\eeqn
where we used the fact that since $\G$ is a connected graph, $\G(\P)$ is also a connected graph. The total number of partitions of the set $\{1,2,\cdots, 2pl\}$ is bounded by $(2pl)!$ and we get
\beqn
\ee \left| \sumk_{i_1\cdots i_l} h_{k i_1 } \tilGk_{i_1, i_2} h_{k i_2 } \cdots h_{k i_l } \tilGk_{i_l j } \right|^{2p} \leq C_{l,p} \max_{ \P} \frac{1}{ q^{2lp-2|\P|}} \frac{1}{ ( N \eta )^{|\P|/2}} \leq C_{l,p} \left( \frac{1}{q} + \frac{1}{ ( N \eta)^{1/4} } \right)^{2 l p}.
\eeqn
This yields $\eqref{eqn:hg}$.


Next we prove \eqref{eqn:ghg}.  We have,
\begin{equation}\label{eqn:unpart2}
\ee \left|  \sumk_{i_1\cdots i_l} \tilGk_{i i_1} h_{k i_1} \tilGk_{i_1 i_2} \cdots h_{k i_l } \tilGk_{i_l j } \right|^{2p} = \sumk_{i_1\cdots i_{2lp} } \ee [ X_{i i_1} h_{k i_1} \cdots X_{i_l j } X_{i i_{l+1} } h_{k i_{l+1} } \cdots X_{i_{2lp} j } ].
\end{equation}
where again we denote by$X_{ij}$ either $\tilGk_{ij}$ or its complex conjugate $(\tilGk_{ij})^*$.
We divide the sum up into partitions $\P$ of $\{ 1, 2, \cdots , 2lp \}$. We define the graph $\G$ representing the unpartitioned sum above as follows. It has $2lp$ black vertices and two white vertices if $i \neq j$ and one white vertex if $i=j$.  We create an edge between black vertices $(i, i+1)$ for $i = ln+1, ln+2, \cdots, (n+1)l-1$ and $n=0, \cdots, 2p-1$, an edge between the white vertex $i$ and black vertices $nl+1$, for $n=0,1,\cdots, 2p-1$, and an edge between the white vertex $j$ and black vertices $nl$ for $n=1, 2,\cdots , 2p$.  
We have that \eqref{eqn:unpart2} is bounded by
\beqn
\ee \left|  \sumk_{i_1, ..., i_l} \tilGk_{i i_1} h_{k i_1} \tilGk_{i_1 i_2} h_{k i_2} \cdots h_{k i_l } \tilGk_{i_l j } \right|^{2p}  \leq \sum_{\P} \frac{C_\P}{ q^{2lp - 2|\P|} } \V ( \G (\P ) , k )
\eeqn
where again the summation is over all partitions on $[[1, 2lp]]$ with each block of size at least $2$.  Applying \eqref{eqn:white} we obtain
\beqn
\ee \left|  \sumk_{i_1\cdots i_l} \tilGk_{i i_1} h_{k i_1} \tilGk_{i_1 i_2} h_{k i_2} \cdots h_{k i_l } \tilGk_{i_l j } \right|^{2p}  \leq C_{l,p}\max_{\P} \frac{1}{ q^{2lp - 2|\P|} } \frac{1}{ ( N \eta )^{ |\P|/2} } \leq C_{l,p} \left( \frac{1}{q} + \frac{1}{ ( N \eta)^{1/4} } \right)^{2 l p}
\eeqn
and we obtain \eqref{eqn:ghg}. 

Lastly, we prove \eqref{eqn:hgh}.  The $2p$-th moment is
\beqn
\ee \left| \sumk_{i_1\cdots i_l} h_{k i_1} \tilGk_{i_1 i_2} h_{k i_2} \cdots \tilGk_{i_{l-1} i_l } h_{k i_l } \right|^{2p} = \sumk_{i_1 \cdots i_{2lp} } \ee [ h_{k i_1} X_{i_1 i_2} \cdots h_{k i_l} h_{k i_{l+1} } X_{ i_{l+1} i_{l+2} } \cdots X_{ i_{2lp-1} i_{2lp} } h_{ k i_{2lp} } ].
\eeqn
Again we divide the sum up into partitions on $\{1,2, \cdots, 2lp \}$. We associate the following graph $\G$ to the unpartitioned sum.  $\G$ has $2lp$ black vertices and no white vertices.  There is an edge between $(i, i+1)$ for $i=nl+1,\cdots, (n+1)l-1$ and $n=0, .., 2lp-1$.  Construct the graph $\G ( \P )$ as above.  We have
\begin{equation}\label{eqn:nowhite_exp}
\ee \left| \sumk_{i_1\cdots i_l} h_{k i_1} \tilGk_{i_1 i_2} h_{k i_2} \cdots \tilGk_{i_{l-1} i_l } h_{k i_l } \right|^{2p} \leq \sum_{\P} \frac{C_\P}{ q^{ 2 l p- 2 |\P|}} \V ( \G ( \P ), k ).
\end{equation}
Let $\G_1, .., \G_r$ be the connected components of $\G(\P)$. Since $\G$ has $2p$ connected components we have $r\leq \min\{|\P|,2p\}$. Therefore by \eqref{eqn:nowhite}, we have
\beqn
\V ( \G(\P), k) = \prod_{i} \V ( \G_i, k) \leq C^{2p(l-1)} \prod_i \frac{1}{ ( N \eta )^{( |\G_i | -1)/2 } }\leq  \frac{C^{2p(l-1)}}{ ( N \eta )^{ (|\P| - r)/2 }},
\eeqn
where $|\G_i|$ is the number of vertices of graph $\G_i$. We then see that \eqref{eqn:nowhite_exp} is  bounded by
\begin{align}
\ee \left| \sumk_{i_1, ..., i_l} h_{k i_1} \tilGk_{i_1 i_2} h_{k i_2} \cdots \tilGk_{i_{l-1} i_l } h_{k i_l } \right|^{2p}  &\leq \sum_{\P }\frac{C_\P}{ q^{ 2 pl -2 |\P|} } \frac{C^{2p(l-1)}}{ ( N \eta )^{ (|\P| - r )/2 } } \notag \\
&\leq C_{l,p} \left( \frac{1}{q} + \frac{1}{ ( N \eta)^{1/4} } \right)^{ 2p (l-2) }.
\end{align}
This completes the proof.
\end{proof}

\begin{proof}[Proof of Proposition \ref{p:ld2}] The bounds \eqref{eqn:hgl1}, \eqref{eqn:ghgl1} and \eqref{eqn:hghl2} follow from more general statements \cite[Lemma A.1]{EKYYSparse1}. 
We only prove \eqref{eqn:hghl3} here. The proof is similar to that of \eqref{eqn:hgh}, but with a more careful analysis of graphs.

The $2p$-th moment  of the lefthand side of \eqref{eqn:hghl3} is bounded by
\begin{align}\label{eqn:unpart_l3}
&\ee \left| \sumk_{i_1, i_2, i_3} h_{k i_1} \tilGk_{i_1 i_2} h_{k i_2} \tilGk_{i_{2} i_3 } h_{k i_3 }  \right|^{2p} \notag \\ 
\leq & \sum_{\P}  \frac{C_{\P}}{ q^{6p-2|\P|}}\left(\frac{1}{N^{|\P|}}\sumk_{\{i_1,i_2\cdots i_{6p}\}\sim\P}\ee [ X_{i_1 i_2} X_{i_2i_3}X_{i_4i_5}\cdots X_{i_{6p-2}i_{6p-1}}X_{i_{6p-1}i_{6p}}]\right),
\end{align}
where the notations are the same as in \eqref{eqn:unpart_bound}, and the sum over $\P$ runs through all the partitions of $\{1,2,\cdots, 6p\}$ where the size of each block is at least two.

We construct the graph $\G$ corresponding to the unpartitioned sum in \eqref{eqn:unpart_l3}.  The graph $\G$ has no white vertices and black vertices labelled by $1, 2, ..., 6p$. We add  edges $(3n+1,3n+2)$ and $(3n+2,3n+3)$ for $n=0,1,2,\cdots, 2p-1$. The graph $\G$ consists of $2p$ connected components and each component is a chain of three vertices and two edges. 

For each partition $\P$ we construct the graph $\G (\P)$ as before. Let $\G_1, \G_2,\cdots, \G_{r}$ be the connected components of $\G(\P)$. Consider the inverse images of the components in the original graph $\G$.  By a relabelling, let the inverse image of each of the components $\G_1, \cdots \G_{r_1}$ consist only of a single chain of $\G$.  The inverse image of the components $\G_{r_1+1}, \cdots \G_r$ consists of at least two chains.  We have the relation 
\begin{align}\label{e:rel_1}
r_1+2(r-r_1)\leq 2p,
\end{align}
since there are in total $2p$ chains in $\G$. 

Each chain consists of three vertices and we know that each block of $\P$ is of size at least two.  Therefore, for any $1\leq i\leq r_1$, the component $\G_i$ consists of a single vertex which corresponds to one block of the partition $\P$ consisting of exactly three numbers. 
Since the other blocks of $\P$ are of at least size two, we have the relation
\begin{align}\label{e:rel_2}
3r_1+2(|\P|-r_1)\leq 6p.
\end{align}
Adding the inequalities \eqref{e:rel_1} and \eqref{e:rel_2} yields,
\begin{align}\label{e:rel_3}
|\P|+r\leq 4p.
\end{align}
The same argument as in the proof of \eqref{eqn:ghg} yields
\beqn
\ee \left|\sumk_{i_1, i_2, i_3} h_{k i_1} \tilGk_{i_1 i_2} h_{k i_2} \tilGk_{i_{2} i_3 } h_{k i_3 }  \right|^{2p}  \leq C_{l,p}\max_{\P }\frac{1}{ q^{ 6p -2 |\P|} } \frac{1}{ ( N \eta )^{ (|\P| - r )/2 } } .
\eeqn
If $|\P|\leq 2p$ then the righthand side of above expression is bounded by $C_{l,p}q^{-2p}$ (recall $r \leq |\P|$). If $|\P|\geq 2p$ then we apply \eqref{e:rel_3} and obtain
\beqn
\frac{1}{ q^{ 6p -2 |\P|} } \frac{1}{ ( N \eta )^{ (|\P| - r )/2 } } \leq \frac{1}{ q^{ 2p+(4p -2 |\P|)} } \frac{1}{ ( N \eta )^{ (2|\P| - 4p )/2 } }\leq \left(\frac{1}{q}+\frac{1}{\sqrt{N\eta}}\right)^{2p}.
\eeqn
This finishes the proof. 
\end{proof}

\section{Universality of Eigenvalue statistics}\label{univsec}
In this section we prove Theorem \ref{t:uni}. We define the following matrix valued process $H_t$ for any random Laplacian-type matrix $H$.  The off-diagonal entries are given by $H_{ij}(t)=h_{ij}(t)$, $i\neq j$ where $h_{ij} (t)$ are independent Ornstein-Uhlenbeck process with initial data $h_{ij}$ satisfying
\begin{align}\label{eq:dynamic}
dh_{ij}(t)=\frac{d B_{ij}(t)}{\sqrt{N}}-\frac{1}{2}h_{ij}(t)dt,\quad 1\leq i<j\leq N+1,
\end{align}
where $\{B_{ij}\}_{1\leq i<j\leq N+1}$ are independent standard Brownian motions. The diagonal elements satisfy $H_{kk} (t) = - \sumk_{i} h_{ik}(t)$.  We have $H_0=H$ and the equality in law
\begin{align}\label{def:Ht}
H_t\stackrel{d}{=}e^{-t/2}H+(1-e^{-t})^{1/2}W,
\end{align}
where $W$ is a Gaussian Laplacian matrix as defined in \eqref{e:GLM}.  Furthermore, $W$ is independent from $H$. 

The matrices $H_t$ all have the trivial eigenvalue $0$ and corresponding eigenvector $\f e$.  This will destroy the eigenvalue universality at $0$ and so we must remove the trivial eigenvalue. 
We denote the nontrivial eigenvalues of $H_t$ as $\lambda_1(t)\leq \lambda_2(t)\cdots \leq \lambda_N(t)$ and corresponding eigenvectors $u_1(t), u_2(t),\cdots, u_N(t)$. We will sometimes omit the time parameter for simplicity of notation. 
We introduce the following Green's function of the $N$ nontrivial eigenvalues and eigenvectors.  It is defined by
\beqn
\hatG (z, H_t):=\sum_{i=1}^{N}\frac{u_i(t)u_i(t)^*}{\lambda_i(t)-z}=R(R^*H_tR-z)^{-1}R^* ,
\eeqn
where $R$ is the matrix defined in Section \ref{s:toy}.
\begin{remark}\label{r:dif}
Clearly $\hatG (z, H_t)$ differs from the usual Green's function by only a rank one matrix,
\begin{align}  \label{e:dif}
\hatG(z, H_t)-(H_t-z)^{-1}=\frac{1}{z} \f e \f e^*.
\end{align}
\end{remark}
We will need the following lemma.
\begin{lemma} \label{lem:Htdecomp}
The nontrivial eigenvalues of $H_t$ are distributed as the eigenvalues of the $N \times N$ matrix
\beqn
\hatH_t := \left( \e^{-t/2} R^* H R + ( 1 - \e^{-t} )^{1/2} ( R^* D R + g I_N ) \right) + ( 1 - \e^{-t} )^{1/2} GOE
\eeqn
where GOE represents an independent $N \times N$ matrix drawn from the Gaussian orthogonal ensemble, $D=\diag\{d_1,d_2,\cdots, d_{N+1}\}$ is a diagonal matrix, with $d_i$'s independent Gaussian random variables with variance $(N+1)/N$, and $g$ is a Gaussian random variable with variance $1/N$.  Denote the matrix
\beqn
\hatA_t := R^* ( \e^{-t/2} H + ( 1 - \e^{-t} )^{1/2} D )R + ( 1 - \e^{-t} )^{1/2} g I_N.
\eeqn
We then have with overwhelming probability uniformly for $z\in \D_L$,
\beq \label{eqn:Alclaw}
\left| m_{\hatA_t} (z) - \mfc (z) \right| \leq \xi \Psi,
\eeq
given that $t\leq \xi/q$, where the control parameter $\Psi$ is introduced in \eqref{controlPsi}.
\end{lemma}
\proof The first statement holds by Proposition \ref{p:decom_G}.  For the second statement, consider the auxillary matrixs
\beqn
A^{(1)} :=\e^{-t/2} H + (1 - \e^{-t} )^{1/2}D,\qquad A^{(2)} :=\e^{-t/2} H + (1 - \e^{-t} )^{1/2}( D+gI_{N+1}).
\eeqn
Since $\hatA_t=R^*A^{(2)}R$, which is the upper left $N\times N$ minor of $\hat{R}^* A^{(2)}R$, from the deterministic bound \cite[Lemma 6.5]{BGMCLTHeavyTail}, the Stieltjes transforms of empirical eigenvalue distributions of $A^{(2)}$ and $\hatA_t$ differ by at most $(N\eta)^{-1}$. Moreover, the Stieltjes transform of the empirical eigenvalue distribution of $A^{(2)}$ is a shift of that of $A^{(1)}$: 
\beqn
 m_{A^{(2)}} (z) =m_{A^{(1)}} (z-(1-e^{-t})^{1/2}g) .
\eeqn
Therefore, to prove \eqref{eqn:Alclaw}, we need to prove that uniformly for $z\in \D_L$ and $t\leq \xi/q$ we have with overwhelming probability,
\beq \label{newtoP}
\left|m_{A^{(1)}} (z-(1-e^{-t})^{1/2}g) - \mfc (z) \right| \leq \xi \Psi .
\eeq
For this, we claim that the proof of Theorem \ref{t:lclaw} holds for $A^{(1)}$.   We summarize here the changes.  We need to modify Proposition \ref{sum}.  Instead of \eqref{e:sum} we see that we obtain, by the same proof using the McDiarmid inequality, (the extra term $(t/N)^{1/2}$ on the righthand side results from the $d_k$'s) the bound
\beqn
\left| \frac{1}{N} \sum_k \frac{1}{ \sumk_i \e^{-t/2} h_{ik} + (1 - \e^{-t} )^{1/2} d_k- z - m_{\fc}^{(1)} (z) } - m_{\fc}^{(1)} (z) \right| \leq C\xi\left(\frac{1}{q}+\frac{\sqrt{t}}{\sqrt{N}}\right) ,
\eeqn
where $m_{\fc}^{(1)}(z)$ is the free convolution of the semicircle with a Gaussian of variance 
\beqn
 \EE\left[\left(\sumk_i \e^{-t/2} h_{ik} + (1 - \e^{-t} )^{1/2} d_k\right)^2\right]=1+\frac{1-e^{-t}}{N}.
\eeqn
More explicitly, $m_{\fc}^{(1)}$ is the unique solution to
\begin{align}\label{e:mfc1}
 m_{\fc}^{(1)}(z)=\int \frac{\rho_G(x)dx}{\left(1+(1-e^{-t})/N\right)^{1/2}x-z-m_{\fc}^{(1)}(z)}.
\end{align}
It's easy to check that the bounds of Lemma \ref{l:reg} also hold for $m_{\fc}^{(1)}$ after changing the constants a little.  

We can then run through the remainder of the proof of Theorem \ref{t:lclaw}, simply by replacing $h_{ij}$ by $e^{-t/2}h_{ij}$, and $G$, $G^{(k)}$ and $\tilde{G}^{(k)}$ by the corresponding quantities associated to $A^{(1)}$. The remainder of the proof is unchanged with the exception of the estimate \eqref{eqn:hghl2}. For the matrix $A^{(1)}$ we instead have the following estimate
\beqn
\1_{\Xik(z)} \sumk_{i,j} \tilGk_{ij}\left(e^{-t}h_{k i} h_{k j} -\delta_{ij}N^{-1}\right) \leq \xi   \left( \frac{1}{q} + \frac{1}{ ( N \eta)^{1/2} } \right)+Ct.
\eeqn
where the constant $C$ depends only on $c_*$ as in \eqref{def:Xi}.
The result of the above argument following the proof of Theorem \ref{t:lclaw} is the following local law for $A^{(1)}$.  We have the estimate,
\beq\label{e:er00}
\left| m_{A^{(1)}} (z)  - m_{\fc}^{(1)} (z) \right| \leq \xi \left(\Psi+(t/N)^{1/2}\right) +Ct
\eeq
with overwhelming probability uniformly for $z \in \D_L$. In order to conclude \eqref{newtoP}, we need an estimate on $|m_{\fc}^{(1)}(z-(1-e^{-t})^{1/2}g)-m_{\fc}(z)|$. For $m_{\fc}^{(1)}$, by \eqref{e:mfc1}, with overwhelming probability uniformly for $z\in \D_L$, we have
\begin{align}
 \notag&m_{\fc}^{(1)}(z-(1-e^{-t})^{1/2}g)\\
 \notag=&\int_{-\xi}^{\xi} \frac{\rho_G(x)dx}{\left(1+(1-e^{-t})/N\right)^{1/2}x+(1-e^{-t})^{1/2}g-z-m_{\fc}^{(1)}(z-(1-e^{-t})^{1/2}g)} +O(N^{-D})\\
 =&\int \frac{\rho_G(x)dx}{x-z-m_{\fc}^{(1)}(z-(1-e^{-t})^{1/2}g)}+O(\frac{\xi\sqrt{t}}{\sqrt{N}}).\label{e:apmfc1}
\end{align}
In the second quality we used the fact that with overwhelming probability $|g|\leq \xi/\sqrt{N}$ and that for $z \in \D_L$ and $-\xi\leq x\leq \xi$ 
the denominator is uniformly bounded away from $0$.  Using \eqref{e:apmfc1} and the defining relation \eqref{e:mfc} of $m_{\fc}(z)$ will yield the following estimate,
\begin{equation}\label{e:er01}
 \left|m_{\fc}^{(1)}\left(z-(1-e^{-t})^{1/2}g\right)-m_{\fc}(z)\right|\leq C\frac{\xi \sqrt{t}}{\sqrt{N}}.
\end{equation}
The proof is similar to the proof of Proposition \ref{p:better}; the argument is essentially the same as that of the proof of Lemma 3.6 of \cite{LSSYdeformWB} and we do not provide all the details.  By a continuity argument, one can first show \eqref{e:er01} for $z\in \D_L$ with $\Im[z]\geq 2$. Since $\mfc^{(1)}$ and $\mfc(z)$ are both Lipschitz and satisfy a similar self-consistent equation, one can prove \eqref{e:er01} for $z$ with slightly smaller imaginary part. By repeating this process, one can finally conclude \eqref{e:er01} for all $z\in \D_L$. 

Combining \eqref{e:er00} and \eqref{e:er01} yields
\beqn
 |m_{\hatA}(z)-m_{\fc}(z)|\leq \xi (\Psi+(t/N)^{1/2})+Ct,
\eeqn
and \eqref{eqn:Alclaw} follows since $t\leq \xi/q$.
\qed

We need the following notion of a derivative of a smooth function $F$ on the space of $(N+1)\times (N+1)$ real symmetric matrices that respects the structure of Laplacian-type matrices.  For $1\leq i<j\leq N+1$ we define the matrices $X_{ij}=E^{ij}+E^{ji}-E^{ii}-E^{jj}$, where $E^{ij}$ is the base matrix given by $(E^{ij})_{kl}=\delta_{ik}\delta_{jl}$.  For indices $1 \leq i < j \leq N+1$ we define
\beqn
\hatdel_{ij}F(H):=\lim_{h\rightarrow 0}\frac{F(H+hX_{ij})-F(H)}{h}.
\eeqn

Let $H$ be a Laplacian-type matrix as in \eqref{e:LM}.  We define the following deformation of $H$.  Given indices $1 \leq i < j \leq N+1$ and number $0 \leq \bm{\theta}^{ij} \leq 1$, 
 the deformed matrix $\bm{\theta}^{ij} H$ is defined by replacing $h_{kl}$ by $\theta^{ij}_{kl}h_{kl}$, for $1\leq k<l\leq N+1$ where $\theta_{kl}^{ij}=1$ unless $\{k,l\}=\{i,j\}$ in which case $\theta_{ij}^{ij}=\theta_{ji}^{ij} = \bm{\theta}^{ij}$.

The proof of Theorem \ref{t:uni} consists of two steps. First we prove that for small $t$, the eigenvalue statistics of the nontrivial eigenvalues of $H$ and $H_t$ are the same by a continuity argument. In the second step, we prove the universality of eigenvalue statistics for random Laplacian-type matrices with small Gaussian components; i.e. we prove eigenvalue universality for $H_t$ with small $t$. Thus, Theorem \ref{t:uni} follows.

\subsection{Short Time Comparison}
In this section we prove that the eigenvalue statistics of $H$ and $H_t$ are the same for small times $t$.
\begin{proposition}\label{p:short_t_comp}
Let $H$ be a random Laplacian-type matrix as in Definition \ref{def:H}. Let $H_t$ by defined as in \eqref{def:Ht}.  Choose $t= N^{-1+\epsilon}$, for sufficiently small $\epsilon\leq \beta/2$, where $\beta$ is from \eqref{defq}.  Then the gap statistics and local correlation functions of the nontrivial eigenvalues of $H_0$ and $H_t$ are the same.  More precisely, we first have
\beq \label{eqn:corcompar}
\lim_{N\to \infty} \int_{\rr^n} O ( \alpha_1, \cdots \alpha_n ) \left\{ \rho_{H_0}^{(n)} \left( E + \frac{ \alpha_1}{N} , \cdots ,  E + \frac{ \alpha_n}{N} \right) - \rho_{H_t}^{(n)} \left( E + \frac{ \alpha_1}{N} , \cdots ,  E + \frac{ \alpha_n}{N} \right) \right\} \d \alpha_1 \cdots \d \alpha_n = 0,
\eeq
for any test function $O \in C_c^\infty ( \rr^n)$.  Fix $\kappa >0$ and $k \in [[ k N , (1 - \kappa) N ]]$.  Then,
\beq \label{eqn:gapcompar}
\lim_{N \to \infty} \ee^{(H_t)} [ O ( N ( \lambda_{k+1}  - \lambda_k ) , \cdots N ( \lambda_{n+k} - \lambda_k ) ) ] - \ee^{(H_0)} [ O ( N ( \lambda_{k+1}  - \lambda_k ) , \cdots N ( \lambda_{n+k} - \lambda_k ) ) ] = 0,
\eeq
for $O \in C_c^\infty (\rr^n)$.
\end{proposition}

We require the following estimate on the resolvent entries.
\begin{proposition}\label{p:reso_bound}
For any $s\geq 0$, we have with overwhelming probability the bound
\begin{align}\label{e:re_bound1}
\max_{ij}\sup_{\bm{\theta}^{kl}}|\hatG_{ij}(z, \bm{\theta}^{kl} H_s)| = O(1)
\end{align}
for $z\in\{z=E+i\eta: |E|\leq L, \xi^3N^{-1}\leq \eta\leq L\}$.
Fix $\gamma >0$.  We have with overwhelming probability uniformly in $s$
\begin{align}\label{e:reb}
\max_{ij}\sup_{\bm{\theta}^{kl}}|\hatG_{ij}(z, \bm{\theta}^{kl} H_s)| = O(\xi^3 N^{\gamma})
\end{align}
for $z\in\{z=E+i\eta: |E|\leq L, N^{-1-\gamma}\leq \eta\leq L\}$.
\end{proposition}
\begin{proof}
We first prove \eqref{e:re_bound1} for the matrix $H_s$ without any deformation.  Since $H_s$ is a random Laplacian-type matrix in the sense of Definition \ref{def:H}, we see from Theorem \ref{t:lclaw} and \eqref{e:dif} that with overwhelming probability,
\beqn
\max_{ij}| \hatG_{ij}(z, H_s)|\leq \max_{ij}|(H_{s}-z)_{ij}^{-1}|+\frac{1}{(N+1)|z|}=O(1)
\eeqn
uniformly for any $z\in\{z=E+i\eta: |E|\leq L, \xi^3 N^{-1}\leq \eta\leq L\}$. For the deformed matrix $\bm{\theta}^{kl} H_s$, we have the resolvent identity,
\beqn
\hatG_{ij}(z,\bm{\theta}^{kl} H_s)=\hatG_{ij}(z, H_s)+(1-\theta_{kl}^{kl})h_{kl}(s)\left( \hatG(z, \bm{\theta}^{kl} H_s) X_{kl} \hatG(z, H_s)\right)_{ij}.
\eeqn
The matrix $X_{kl}$ has exactly four nonzero entries and with overwhelming probability we have $|h_{kl}(s)|\leq \xi/\log (N)q$. For $N$ large enough we therefore have with overwhelming probability,
\beqn
\max_{ij}\sup_{\bm{\theta}^{kl}}| \hatG_{ij}(z,\bm{\theta}^{kl} H_s)|\leq \max_{ij}| \hatG_{ij}(z, H_s)|+\frac{1}{2}\max_{ij}\sup_{\bm{\theta}^{kl}}| \hatG(z, \bm{\theta}^{kl} H_s)|,
\eeqn
and \eqref{e:re_bound1} follows from rearranging the above expression. 
The following deterministic relation is a consequence of \cite[Lemma 2.1]{BKYDRegular},
\beqn
\max_{ij}| \hatG_{ij}(E+i\eta/M, \bm{\theta}^{kl} H_s)|\leq M\max_{ij}| \hatG_{ij}(E+i\eta, \bm{\theta}^{kl} H_s)|.
\eeqn
We obtain \eqref{e:reb} by taking $M=\xi^3 N^{\gamma}$.
\end{proof}

We will need the following lemma which is a slight modification of \cite[Lemma A.2]{BoYa}.
\begin{lemma}\label{L:conti}
Let $H$ be a random Laplacian-type matrix.  
Define $H_t$ as in \eqref{def:Ht}. 
 Let $F$ be a smooth function on the space of $(N+1)\times (N+1)$ real symmetric matrices satisfying
\begin{align}\label{e:contib}
\sup_{0\leq s\leq t, \atop 1\leq i<j\leq N+1}\EE[(N^2|h_{ij}(s)|^3 +N|h_{ij}(s)|)\sup_{\bm{\theta}^{ij}}|\hatdel_{ij}^3F( \bm{\theta}^{ij} H_s)|]\leq B,
\end{align}
where the deformed matrix $\bm{\theta}^{ij} H_s$ is defined as above.  
Then,
\begin{align}{\label{diff}}
|\EE[F(H_t)]-\EE[F(H_0)]|\leq C tB.
\end{align}
\end{lemma}
\begin{proof}
From It\^o's formula we have
\begin{align}
\EE[F(H_t)]-\EE[F(H_0)]
=&\int_0^t \EE[\sum_{i<j}\hatdel_{h_{ij}}F(H_s)\d h_{ij}(s)+\frac{1}{2N}\hatdel_{h_{ij}}^2F(H_s)\d s] \notag \\
=&\int_0^t \EE[\sum_{i<j}-\frac{1}{2}h_{ij}(s)\hatdel_{ij}F(H_s)+\frac{1}{2N}\hatdel_{ij}^2F(H_s)]\d s
\end{align}
We expand the terms above.  For some $0 \leq \bm{\theta}_1^{ij} , \bm{\theta}_2^{ij} \leq 1$ we have the equalities
\beqn
 \EE[h_{ij}(s)\hatdel_{ij}F(H_s)]
 =\EE[h_{ij}\hatdel_{ij}F_{h_{ij}=0}+h_{ij}^2\hatdel_{ij}^2F_{h_{ij}=0}+\frac{1}{2}h_{ij}^3 \hatdel_{ij}^3 F(\bm{\theta}_1^{ij}H_s )] ,
\eeqn
and
\beqn
  \EE[\hatdel_{ij}^2F(H_s)]
 =\EE[\hatdel_{ij}^2F_{h_{ij}=0}+h_{ij}\hatdel_{ij}^3F(\bm{\theta}_2^{ij}H_s )].
\eeqn
This leads to
\beqn
\left|\EE[\sum_{i<j}-\frac{1}{2}h_{ij}(s)\hatdel_{ij}F(H_s)+\frac{1}{2N}\hatdel_{ij}^2F(H_s)]\right|\leq N^2\sup_{i<j}\EE[(|h_{ij}(s)|^3 +\frac{1}{N}|h_{ij}(s)|)\sup_{\bm{\theta}^{ij}}|\hatdel_{ij}^3F( \bm{\theta}^{ij} H_s)|],
\eeqn
and \eqref{diff} follows.
\end{proof}

\subsubsection{Comparison of correlation functions}
We now prove that the locally averaged correlation functions of the nontrivial eigenvalues of the matrix $H_0$ coincide with those of $H_t$, i.e. \eqref{eqn:corcompar} of Proposition \ref{p:short_t_comp}.  
The main ingredient of the proof is the following Green's function comparison lemma. Its proof follows essentially from \cite[Lemma 5.2]{Sparse}.
\begin{lemma} \label{lem:Gcorr}
Fix $n \in \mathbb{N}$, and let $\gamma,\delta > 0$ be sufficiently small.
Then the following holds for any $\eta \in [N^{-1-\gamma}, N^{-1}]$,
any sequence of positive integers $k_1,k_2,\dots,k_n$, any set of complex parameters
$z_j^m = E_j^m\pm i\eta$, where $j \in [[{1,k_m}]]$, $m \in [[{1,n}]]$, $|E_j^m| \leq  L$,
and the $\pm$ signs are arbitrary.
Let $F \in C^\infty(\RR^n)$ be a test function such that for any multi-index $m=(m_1,\cdots, m_n)$
with $1\leq |m|\leq 3$ and for any $\omega>0$,
\begin{align}\label{e:reg}
\begin{split}
\max\left\{|\partial^m F(x)|: \|x\|_\infty \leq N^{\omega}\right\}
&\leq N^{c_0\omega},
\\
\max\left\{|\partial^m F(x)|: \|x\|_{\infty}\leq N^{2}\right\}
&\leq N^{c_0} .
\end{split}
\end{align}
Then, defining $\hatG_1 := \hatG(z, H_t)$ and $\hatG_0:=\hatG(z, H_0)$ we have for any $t>0$,
\begin{multline} \label{e:Gcorr}
  \left|  \EE F\left(N^{-k_1} \Tr \prod_{j=1}^{k_1} \hatG_1(z_j^1), \dots,
      N^{-k_n} \Tr \prod_{j=1}^{k_n}\hatG_1(z_j^n)\right)  -\EE F( \hatG_1\rightarrow \hatG_0) \right|
 \leq C_1(\xi N^{\gamma})^{C_2} tN/q.
\end{multline}
Here, $\EE F(\hatG_1 \rightarrow \hatG_0)$ is the expression obtained from the leftmost expression above by replacing $\hatG_1$ with $\hatG_0$.
The constants $C_1$ and $C_2$ depend on $n$, $k_1, \dots, k_n$, $m_1, \dots, m_n$, and the constants in \eqref{e:reg}.
\end{lemma}
\begin{proof}
For simplicity of notation, we only prove \eqref{e:Gcorr} for $n=1$ and $k_1=1$ and write $z$ instead of $z_1^1$; i.e., we prove that
\beqn
\left| \EE[F(N^{-1}\Tr \hatG(z,H_t))]-\EE[F(N^{-1}\Tr \hatG(z, H_0))]  \right| \leq C_1(\xi N^{\gamma})^{C_2} tN/q.
\eeqn
By Lemma \ref{L:conti}, it is sufficient to prove that
\begin{align}\label{e:up_bound}
\sup_{0\leq s\leq t, \atop 1\leq i<j\leq N+1}\EE[(N^2|h_{ij}(s)|^3 +N|h_{ij}(s)|)\sup_{\bm{\theta}^{ij}}|\hatdel_{ij}^3F(N^{-1}\Tr \hatG(z,\bm{\theta}^{ij}H_s)|]\leq C_1 (\xi N^{\gamma})^{C_2} N/q.
\end{align}
An explicit computation yields
\beqn
\hatdel_{ij}^3 F(N^{-1}\Tr \hatG)=(N^{-1}\hatdel_{ij}\Tr \hatG)^3F'''+3(N^{-1}\hatdel_{ij}\Tr \hatG)(N^{-1}\hatdel_{ij}^2\Tr \hatG)F''+(N^{-1}\hatdel_{ij}^3\Tr \hatG)F'.
\eeqn
The derivative of the trace is
\beqn
\hatdel_{ij}^r \Tr \hatG = (-1)^r r! \Tr [ \hatG (X_{ij} \hatG )^r ].
\eeqn
The matrix $X_{ij}$ consists of four nonzero entries and so $(\hatdel_{ij})^l\Tr G$ is a sum of $r!4^r$ of terms which are products of resolvent entries. By Proposition \ref{p:reso_bound} we have with overwhelming probability for any $z$ with $\Im [z]\geq N^{-1-\gamma}$ that
\begin{align}\label{e:high_p}
\sup_{\bm{\theta}^{kl}}\left|N^{-1}(\hatdel_{ij})^{r}\Tr \hatG(z,\bm{\theta}^{ij}H_s)\right|\leq 4^rr! (\xi^3 N^{\gamma})^{r+1},\quad r=1,2,3.
\end{align}
We denote $\cal A$ the event that \eqref{e:high_p} holds. On $\cal A$ we have
\beqn
\sup_{\bm{\theta}^{ij}}|\hatdel_{ij}^3F(N^{-1}\Tr \hatG(z,\bm{\theta}^{ij}H_s)|\leq C (\xi^3 N^\gamma)^{6+c_0}.
\eeqn
On the complement $\A^c$ we have the deterministic bound
\beqn
\sup_{\bm{\theta}^{kl}}\left|N^{-1}(\del_{ij})^{r}\Tr \hatG(z,\bm{\theta}^{ij}H_s)\right|\leq r! (4/\Im [z])^r\leq 4^rr!(N^{1+\gamma})^{r+1},\quad r=1,2,3.
\eeqn
We deduce \eqref{e:up_bound}, yielding the lemma.
\end{proof}

The following lemma is essentially \cite[Theorem 6.4]{erdos2012bulk}.
It transforms the statement about the Green's function of Lemma~\ref{lem:Gcorr}
to a statement about the local correlation functions. 

\begin{lemma}\label{lem:comparecorr}
Consider two random Laplacian-type matrices $H_0$ and $H_1$ with Green's functions $\hatG_0(z)$ and $\hatG_1(z)$.
Suppose that for all choices of the parameters (i.e., $n, \gamma, \delta$, $k_i$'s, $z_i$'s, etc. ) and functions $F$ the estimate \eqref{e:Gcorr} holds.  Then for all smooth observables $O$
\beqn
\lim_{N\to \infty} \int_{\rr^n} O ( \alpha_1, \cdots \alpha_n ) \left\{ \rho_{H_0}^{(n)} \left( E + \frac{ \alpha_1}{N} , \cdots ,  E + \frac{ \alpha_n}{N} \right) - \rho_{H_1}^{(n)} \left( E + \frac{ \alpha_1}{N} , \cdots ,  E + \frac{ \alpha_n}{N} \right) \right\} \d \alpha_1 \cdots \d \alpha_n = 0.
\eeqn
\end{lemma}
We have therefore proven \eqref{eqn:corcompar} of Proposition \ref{p:short_t_comp}.
\subsubsection{Comparison of gap statistics}

In this subsection, we prove \eqref{eqn:gapcompar} of Proposition \ref{p:short_t_comp}, the comparison of the gap statistics for random Laplacian-type matrices, which is extremely similar to the argument given in detail in \cite[Section 4, 5]{Sparse}.  The key difference is that we do not have a delocalization estimate for all $N$ eigenvectors but only for the eigenvectors corresponding to eigenvalues belonging to the interval $[-L, L]$.

The first step is to obtain the following level repulsion estimate for random Laplacian-type matrices.

\begin{proposition} \label{p:tlevp}
Let $H$ be a random Laplacian-type matrix. 
Fix a small $\kappa>0$ and fix an index $i \in [[ \kappa N, ( 1 - \kappa ) N ]]$.  For all sufficiently small $\tau >0$ we have
\begin{align}\label{e:levp}
\PP(|\lambda_i(H)-\lambda_{i+1}(H)|\leq N^{-1-\tau})\leq N^{-\tau/2}.
\end{align} 
\end{proposition}
\begin{proof}  The proof is similar to the proof of \cite[Theorem 4.1]{Sparse}.  We detail only what changes.  Denote the matrix
\beqn
R_i(H) :=\sum_{j:j\neq i,j\leq N}\frac{1}{\lambda_i(H)-\lambda_j(H)}u_j(H)u_j(H)^*=\oint_{|z-\lambda_i(H)|=\omega}\frac{\hatG(z, H)}{\lambda_i(H)-z}\d z,
\eeqn
where $\omega$ is chosen such that the contour $|z-\lambda_i(H)|=\omega$ encloses only $\lambda_i(H)$. We define the quantity
\beqn
Q_i(H):=\frac{1}{N^2}\Tr (R_i(H)^2)=\frac{1}{N^2}\sum_{j:j\neq i,j\leq N}\frac{1}{(\lambda_j(H)-\lambda_i(H))^2}.
\eeqn
This quantity plays an important role in \cite{tao2011,tao2010}, where it was observed that it captures quantitatively
the derivatives of $\lambda_i(H)$. The function $Q_i(H)$ is not well-defined on the space of real symmetric matrices (it diverges when $\lambda_i(H)$ is not a single eigenvalue) and so we compose it with a cutoff function $\chi = \chi_N$.  The function $\chi$ is chosen so that: (1) it is smooth with its first three derivatives bounded; (2) 
on the interval $[0,N^{2\tau}]$ we have that $|\chi(x)-x|\leq 1$, and  for $x\geq N^{2\tau}$, $\chi(x)=N^{2\tau}$.  The function $\chi (Q_i (H))$ is then a well-defined smooth function on the space of real symmetric matrices.


The proof of Proposition \ref{p:tlevp} consists of three steps. The first step is to obtain the estimate
\begin{align} \label{e:step1}
\EE[\chi(Q_i(H_t))]\leq CN^{3\tau/2},
\end{align}
for some $t=N^{-1+\epsilon}$ with $\epsilon$ sufficiently small. This estimate follows from \cite[Theorem 3.6]{DBM} and Lemma \ref{lem:Htdecomp}.

In the second step, we compare $\EE[\chi(Q_i(H_t))]$ for $t=N^{-1+\epsilon}$ and $\EE[\chi(Q_i(H_0))]$ using the continuity Lemma \ref{L:conti}.  By choosing $\epsilon$ sufficiently small we claim that
\begin{align}\label{e:step2}
|\EE[\chi(Q_i(H_t))]-\EE[\chi(Q_i(H_0))]|\leq 1.
\end{align}
Taking $F(H):= \chi(Q_i(H))$ in Lemma \ref{L:conti} we see that it suffices to prove that
\begin{align}\label{e:derb}
\sup_{0\leq s\leq t, \atop 1\leq a<b\leq N+1}\EE[(N^2|h_{ab}(s)|^3 +N|h_{ab}(s)|)\sup_{\bm{\theta}^{ab}}|\hatdel_{ab}^3\chi(Q_i(\bm{\theta}^{ab}H_s))|]\leq C_1(\xi N^{\tau})^{C_2}N/q.
\end{align}
The proof of this bound is similar to the computation completed in \cite[Proposition~4.6]{Sparse}.  We outline the differences. The formulas \cite[(4.16)--(4.18)]{Sparse} remain valid after replacing $V$ by the $X_{ab}$.  The end result is that $\hatdel_{ab}^3 Q_i (H)$ is a finite sum of terms of the form 
\begin{align}
&\frac{1}{N^2}  \sum_{j: j\neq i}\frac{\Omega}{(\lambda_i-\lambda_j)^5},\quad \Omega=V_{ii}^3, V_{ii}^2V_{jj}, V_{ii}V_{jj}^2, V_{jj}^3.\notag \\
&\frac{1}{N^2}  \sum_{j_1,j_2:\atop j_1,j_2\neq i}\frac{\Omega}{(\lambda_i-\lambda_{j_1})^4(\lambda_{i}-\lambda_{j_2})},\quad \Omega= V_{ii}V_{j_1j_2}^2,V_{j_1j_1}V_{ij_2}^2,V_{j_1j_1}V_{j_1j_2}^2, V_{ii}V_{ij_2}^2.\notag \\
&\frac{1}{N^2}  \sum_{j_1,j_2:\atop j_1,j_2\neq i}\frac{\Omega}{(\lambda_i-\lambda_{j_1})^3(\lambda_{i}-\lambda_{j_2})^2},\quad
\Omega= V_{j_2j_2}V_{j_1j_2}^2, V_{ii}V_{ij_1}^2, V_{ii}V_{j_1j_2}^2, V_{j_1j_1}V_{j_1j_2}^2.\notag \\
&\frac{1}{N^2}  \sum_{j_1,j_2,j_3:\atop j_1,j_2,j_3\neq i}\frac{V_{j_1j_2}V_{j_2j_3} V_{j_3j_1}}{(\lambda_i-\lambda_{j_1})^3(\lambda_i-\lambda_{j_2})(\lambda_i-\lambda_{j_3})},
\sum_{j_1,j_2,j_3:\atop j_1,j_2,j_3\neq i}\frac{V_{j_1j_2}V_{j_2j_3} V_{j_3j_1}}{(\lambda_i-\lambda_{j_1})^2(\lambda_i-\lambda_{j_2})^2(\lambda_i-\lambda_{j_3})}, \label{eqn:delabQterms}
\end{align}
where $V_{jk} := u_j^* X_{ab} u_k$.  The only change in controlling the above quantities arises from the fact that we do not have the delocalization of eigenvectors corresponding to eigenvalues outside the interval $[-L, L]$. However, 
the contribution for eigenvalues far away from $\lambda_i(\bm{\theta}^{ab}H_s)$ is negligible due to the fact that the denominator is bounded below.   For the sake of exposition, we show how the existing proof of \cite[Section 4]{Sparse} is modified to bound the term
\beqn
\frac{1}{N^2} \sum_{j : j \neq i}  \frac{ V_{jj}^3}{ ( \lambda_i - \lambda_j)^5}.
\eeqn
On the event  $\{Q_i(\bm{\theta}^{ab}H_s)\leq N^{2\tau}\}$ we have
\beqn
\sum_{j : j \neq i} \frac{1}{ | \lambda_i - \lambda_j |^k} \leq \left( \sum_{j : j \neq i} \frac{1}{ | \lambda_i - \lambda_j|^2} \right)^{k/2} \leq N^{k ( 1 + \tau) }, \qquad k \geq 2.
\eeqn
 Let $\A$ be the event that the bound \eqref{e:re_bound1} holds.  On the event $\A$ we have that
\beqn
\|u_j\|_{\infty}^2\leq \xi^3/N, \qquad j : |\lambda_j-\lambda_i|\leq L/2.
\eeqn
Therefore on the event $\A \cap \{Q_i(\bm{\theta}^{ab}H_s)\leq N^{2\tau}\}$ we have
\begin{align}
\frac{1}{N^2}  \sum_{j : j \neq i}  \frac{ V_{jj}^3}{ | \lambda_i - \lambda_j |^5} & \leq \frac{1}{N^2}  \sum_{{ j : j \neq i \atop |\lambda_j - \lambda_i | \leq L/2}}  \frac{ V_{jj}^3}{ | \lambda_i - \lambda_j |^5} + \frac{1}{N^2}  \sum_{j : j \neq i \atop |\lambda_j - \lambda_i | \geq L/2}  \frac{ V_{jj}^3}{ | \lambda_i - \lambda_j |^5} \notag \\
&\leq \frac{\xi^9}{N^5} \sum_{j : j \neq i} \frac{1}{ | \lambda_i - \lambda_j|^5} + \frac{1}{N^2}  \frac{2}{L}  \sum_{j : j \neq i}  | u_j^* X_{ab} u_j|^3 \notag \\
& \leq \xi^9 N^{5 \tau} + \frac{C}{N^2} \leq C_1 ( \xi N^{\tau} )^{C_2}
\end{align}
The other terms of \eqref{eqn:delabQterms} are handled similarly.  The bound \eqref{e:derb} is subsequently obtained by a similar argument as in Lemma \ref{lem:Gcorr}. 
We then  choose $\tau$ and $\xi$ sufficiently small such that $C_1(\xi N^{\tau})^{C_2}\leq N^{\beta/4}$, where $\beta$ is from \eqref{defq}, and $t=N^{-1+\epsilon}$ with $\epsilon\leq \beta/2$, then we have  $t C_1(\xi N^{\tau})^{C_2}N/q\leq N^{-\beta/4}$ and \eqref{e:step2} follows.

Combining \eqref{e:step1} and \eqref{e:step2} yields
\beqn
\EE[\chi(Q_i(H_0))]\leq CN^{3\tau/2}.
\eeqn
The estimate \eqref{e:levp} follows from Markov's inequality and the definition of $\chi$.
\end{proof}

\begin{proof}[Proof of \eqref{eqn:gapcompar} of Proposition \ref{p:short_t_comp}] For any test function $O\in C^{\infty}(\RR^n)$ with bounded first three derivatives, it suffices to prove the stronger claim:
\begin{align}\label{e:stat_coin}
\lim_{N\rightarrow \infty}\left|\EE[O(N\lambda_i(0),\cdots, N\lambda_{i+n}(0) )] -\EE[O(N\lambda_i(t),\cdots, N\lambda_{i+n}(t ) )]\right|=0
\end{align}
for $t\leq N^{-1+\epsilon}$. For simplicity of notion we only prove \eqref{e:stat_coin} in the case $n=1$:
\begin{align}\label{e:case1}
\lim_{N\rightarrow \infty}|\EE[O(N\lambda_i(t))]-\EE[O(N\lambda_i(0))]|=0.
\end{align}
The general case is analogous. 
 
 Take a cutoff function $\rho$ such that $\rho(x)=1$ for $x\leq N^{2\tau}$ and $\rho(x)=0$ for $x\geq 2N^{2\tau}$, where $\tau>0$ is a small constant.  We can choose $\rho$ so that its first three derivatives are bounded. We have by Proposition \ref{p:tlevp} that
\beqn
\PP(Q_i(H_s)\geq N^{2\tau})\leq N^{-\tau/2}, \quad s=0, t.
\eeqn
Since $O\in C^{\infty}(\RR)$ is bounded,  we have
\begin{align}
&\left|\EE[O(N\lambda_{i}(0))]-\EE[O(N\lambda_{i}(t))]\right| \notag \\
\leq &\left|\EE[O(N\lambda_{i}(0))\rho_{M}(Q_i(H_0))]-\EE[O(N\lambda_{i}(t))\rho_{M}(Q_i(H_t))]\right|  \notag \\
+ &\|O\|_{\infty}\PP(Q_i(H_0)\geq N^{2\tau}))+\|O\|_{\infty}(\PP(Q_i(H_t)\geq N^{2\tau})) \notag \\
\leq &\left|\EE[O(N\lambda_{i}(0))\rho_{M}(Q_i(H_0))]-\EE[O(N\lambda_{i}(t))\rho_{M}(Q_i(H_t))]\right|+\frac{2\|O\|_{\infty}}{N^{\tau/2}}.
\end{align}
Notice that $O(N\lambda_{i}(H))\rho_{M}(Q_i(H))$ is a well-defined smooth function on the space of symmetric matrices. We apply Lemma \ref{L:conti} with the function $F(H)=O(N\lambda_{i}(H))\rho_{M}(Q_i(H))$.  An argument similar to the derivation of \eqref{e:derb} yields
\beqn
\sup_{\bm{\theta}^{ab}}|\hatdel_{ab}^3O(N\lambda_{i}(H))\rho_{M}(Q_i(H))|\leq C_1(\xi N^{\tau})^{C_2},
\eeqn
with overwhelming probability. Thus by Lemma \ref{L:conti},
\beqn
\left|\EE[O(N\lambda_{i}(0))\rho_{M}(Q_i(H_0))]-\EE[O(N\lambda_{i}(t))\rho_{M}(Q_i(H_t))]\right|\leq C_1(\xi N^{\tau})^{C_2}tN/q.
\eeqn
The equation \eqref{e:case1} follows if we hoose $\tau$ and $\xi$ sufficiently small such that $C_1(\xi N^{\tau})^{C_2}\leq N^{\beta/4}$, where $\beta$ is from \eqref{defq}, and $t=N^{-1+\epsilon}$ with $\epsilon\leq \beta/2$, then we have  $t C_1(\xi N^{\tau})^{C_2}N/q\leq e^{-\beta/4}$. And \eqref{e:case1} follows.
\end{proof}
We have therefore proven \eqref{eqn:gapcompar} of Proposition \ref{p:short_t_comp}. This finishes the proof of Proposition \ref{p:short_t_comp}.

\subsection{Universality of $H_t$}
The goal of this section is the following proposition.
\begin{proposition}\label{p:uniHt}
Let $H$ be a random Laplacian-type matrix. 
We define $H_t$ as  in \eqref{def:Ht}. Let $t = N^{-1+\epsilon}$, for sufficiently small $\epsilon\leq \beta/2$, where $\beta$ is from \eqref{defq}.  There exists a constant $c$ so that the following holds.  Firstly, we have gap universality for $H_t$.  For any $\kappa>0$ and index $i\in [[\kappa N, (1-\kappa)N]]$ we have
\begin{align}
\begin{split}
&\left|\EE^{(H)}[ O(N \rfc(\gamma_i)(\lambda_i-\lambda_{i+1}), \cdots, N\rfc(\gamma_i)(\lambda_i -\lambda_{i+n}))]\right.\\
-&\left.\EE^{(GOE)}[ O(N \rho_{sc}(\gamma_{i,sc})(\lambda_i-\lambda_{i+1}), \cdots, N\rho_{sc}(\gamma_{i,sc})(\lambda_i -\lambda_{i+n}))]\right|\leq C N^{-c},
\end{split}
\end{align}
for any test function $O \in C_c^\infty ( \rr^n)$. 

 Secondly, the averaged $n$-point correlation functions of the nontrivial eigenvalues of $H_t$ are universal. For any $\delta >0$ and $E\in [-L,L]$,  and $b\geq N^{-1+\delta}$ we have
\begin{align}
\begin{split}
&\left|\int_{E-b}^{E+b}\int_{\RR^n} O(\alpha_1,\cdots, \alpha_n)\left\{ \frac{1}{\rfc(E)^n}\rho^{(n)}_{H_t}\left( E'+\frac{\alpha_1}{N\rfc(E)},\cdots E'+\frac{\alpha_n}{N\rfc(E)}\right)\right.\right.\\
-&\left.\left.\frac{1}{\rho_{sc}(E)^n}\rho^{(n)}_{GOE}\left( E'+\frac{\alpha_1}{N\rho_{sc}(E)},\cdots E'+\frac{\alpha_n}{N\rho_{sc}(E)}\right)\right\} \d \alpha_1 ... \d \alpha_n \frac{\d E'}{2b}\right|\leq N^{-c},
\end{split}
\end{align}
where the test function $O\in C^{\infty}_{c}(\RR^n)$.
\end{proposition}

The proof of Proposition \ref{p:uniHt} is a modification of \cite[Section 3]{Sparse}.  We do not provide all of the details.

We first require a minor modification of the main result of \cite[Theorem 2.5]{DBM}, which states that  for any deterministic matrix $A$ the local eigenvalue statistics of $A + \sqrt{t} GOE$ are universal for $t \gtrsim N^{-1}$, under some mild regularity conditions of its eigenvalues.  To state it more precisely, we introduce some notation. For the deterministic $N\times N$ matrix $A$ we define
\beqn
A_t := A + \thetat G,\qquad \thetat := \sqrt{\frac{1 - e^{-t}}{2} } = O ( t^{1/2} ) ,
\eeqn
where $G$ is a standard GOE matrix.  We denote by $\rho_t$ the free convolution of the empirical eigenvalue distribution of $A$ and the semicircle law with variance $\thetat^2$, and  $m_t$ the Stieltjes transform of $\rho_t$. Then $m_0 (z)$ is just the Stieltjes transform of empirical eigenvalue distribution of $A$. The density $\rho_t$ is analytic on its support for any $t>0$.  The function $m_t$ solves the equation
 \beqn
m_t (z) = m_0 ( z  + \thetat^2 m_t (z) ) = \frac{1}{N} \sum_{i=1}^N \frac{1}{ \lambda_i (A) - z - \thetat^2 m_t (z) },
\eeqn
where $\lambda_1(A)\leq \lambda_2(A)\leq\cdots \leq \lambda_N(A)$ are the eigenvalues of $A$.
We denote the classical eigenvalues of $\rhofc$, $\rhosc$ and $\rho_t$ by $\gamma_i$, $\gamma_{i,sc}$  and $\gamma_{i, t}$ respectively.  They are defined by
\beqn
\int_{-\infty}^{\gamma_i}\rhofc(x)dx=\frac{i-\frac{1}{2}}{N},\qquad
\int_{-\infty}^{\gamma_{i,sc}}\rhosc(x)dx=\frac{i-\frac{1}{2}}{N},\qquad
\int_{-\infty}^{\gamma_{i,t}}\rho_{t}(x)dx=\frac{i-\frac{1}{2}}{N}. 
\eeqn

The following is a minor modification of \cite[Theorem 2.5]{DBM}.

\begin{theorem}  \label{thm:abulk} Fix small parameter $\epsilon > 0$ and large parameter $L>0$. Suppose that there are constants $C_1,C_2 >0$ such that
\beqn
C_1 \leq \mathrm{Im } [ m_{0} (E + \i \eta ) ] \leq C_2
\eeqn
for all $E \in (-L, L)$ and $N^{-1+\epsilon/3} \leq \eta \leq L$.  There is a constant $c_\eps$ so that the following holds.  Suppose that $\|A\| \leq N^{c_\eps}$.  Let $i$ be such that $\lambda_i (A) \in (-L/2, L/2)$ and index $j\in [[\kappa N, (1-\kappa)N]]$ for some $\kappa>0$. Then for $t=N^{-1+\epsilon}$, there exists a small constant $c>0$, which depends on $\epsilon$ and $\kappa$, such that for indices $i_1, ..., i_n \in \nn$ and $i_k \leq N^{c}$,
\begin{align}
\big| &\EE^{(A_t)} [ O (  \rho_t ( \gamma_{i,t} ) N ( \lambda_i - \lambda_{i+i_1} ) , \cdots ,  \rho_t ( \gamma_{i,t} )N ( \lambda_i - \lambda_{i + i_n } ) ) ]\notag \\
 -   &\EE^{(GOE)} [ O ( \rhosc (\gamma_j )N ( \lambda_j - \lambda_{j+i_1} ) , \cdots ,  \rhosc ( \gamma_j )N ( \lambda_j - \lambda_{j + i_n } ) ) ] \big| \leq C N^{-c}.
\end{align}
\end{theorem}
The only difference between the above theorem and that stated in \cite[Theorem 2.5]{DBM} is that the spectrum of the matrix $A$ is required to have a lower bound in \cite[Theorem 2.5]{DBM}.  However, the only place the lower bound is used is to apply \cite[Lemma 7.17]{DBM} in order to conclude \cite[Theorem 3.5]{DBM}; that is, so that the error in the rigidity estimates on the eigenvalues of $A_t$ is smaller than $N^\delta/N$ for any $\delta >0$.  It is clear that if instead the lower bound on $A$ is replaced by the bound $||A|| \leq N^{\delta_1}$ then there is a universal $C >0$ so that the rigidity estimates of \cite[Theorem 3.5]{DBM} hold with an error of $N^{C \delta}/N$.  The proof of the main result of \cite{DBM} on universality only requires that the rigidity estimates hold with an error of $N^{\delta_2} /N$ where $\delta_2 >0$ is some small constant.  Hence Theorem \ref{thm:abulk} holds.

Let $\A$ be the set of $N\times N$ real symmetric matrices such that
\begin{equation} \label{defA}
\A := \{ A : | m_0 (z) -\mfc (z) | \leq \xi\Psi) , z \in \D_L\} \cap \{ \|A\| \leq \xi\} ,
\end{equation}
where the $\nu >0$ in the definition of $\xi$ is taken small enough.


\begin{lemma} \label{lem:mbds}
Let $A \in \A$.  We have for $t \leq q/N$, (in fact we only need that $t=o(1)$)
\beqn
| m_t (z) - \mfc (z) | \leq C (\xi\Psi+t^{1/2})
\eeqn
uniformly for 
\begin{align}
z\in &\{ z =E+i\eta : |E| \leq 2\xi+L-1, 2 \leq \eta \leq L-1 \}, \notag \\
\cup &\{ z = E + \i \eta : |E| \leq L-1, \xi^3N^{-1} \leq \eta \leq 2 \}, \notag \\
\cup &\{ z = E + \i \eta : 2\xi+1 \leq |E| \leq 2\xi + L -1,  \xi^3N^{-1} \leq \eta \leq 2 \}.
\end{align}
\end{lemma}
The proof is identical to  \cite[Lemma 3.3]{Sparse}. We first derive that $|\thetat^2 m_t|\leq \thetat=O(t^{1/2})$, and then use the definition of the set $\A$ to get
\beqn
m_t(z)=m_0(z+\thetat^2m_t(z))=\mfc(z+\thetat^2m_t(z))+O(\xi\Psi)=\mfc(z)+O(\xi\Psi+t^{1/2}),
\eeqn
where we have used that $\mfc$ is Lipschitz from Lemma \ref{l:reg}.

We now need the following analogue of \cite[Lemma 3.4]{Sparse}.

\begin{lemma} \label{lem:determbds}
For any $t\leq \min\{q/N,N^{-2/3}\}$, any real symmetric matrix $A\in \A$ and $E\in (-L/2, L/2)$ we have
\begin{align}
\label{detE}\left| \rho_t ( E) - \rhofc ( E ) \right| \leq C \xi^3/(Nt).
\end{align}
Moreover for any index $i$ such that $\lambda_i(A)\in (-L/2, L/2)$, we have 
\begin{align}\label{det_rho}
\left| \rho_t ( \gamma_{i,t} ) - \rhofc ( \gamma_i ) \right| \leq C \xi^3/(Nt).
\end{align}
\end{lemma}
\begin{proof} The proof is similar to that of \cite[Lemma 3.4]{Sparse}.  We have the bounds
\beqn
\rho_t(E) \leq C,\qquad  | \rho_t'(E) | \leq C/t, 
\eeqn
uniformly for $E\in [-L,L]$. For \eqref{detE}, we can write the difference as
\begin{align}
|\rho_t ( E) - \rhofc ( E )|
\leq& |\rho_t(E)-\Im [m_t(E+\i\eta)]/\pi| \notag \\
+&|\Im [m_t(E+\i\eta)]/\pi-\Im [\mfc(E+\i\eta)/\pi]|+|\Im [\mfc(E+\i\eta)]/\pi-\rhofc(E)|.
\end{align}
The statement follows by a similar argument as in \cite[Lemma 3.4]{Sparse} and taking $\eta=\xi^3/N$.

The proof of \eqref{det_rho} is analogous to the proof of \cite[Lemma 3.4]{Sparse}, and the main ingredient is the rigidity estimate
\beqn
|\gamma_{i, t} - \gamma_i |\leq C(\xi^2q^{-1}+\xi t^{1/2}),
\eeqn
which follows the same argument as the proof of Corollary \ref{cor:rig} and the definition of $\cal A$. To prove \eqref{det_rho} we have
\beqn
\left| \rho_t ( \gamma_{i,t} ) - \rhofc ( \gamma_i ) \right| 
\leq \left| \rho_t ( \gamma_{i,t} ) - \rhofc ( \gamma_{i,t} ) \right|+ \left| \rhofc ( \gamma_{i, t} ) - \rhofc ( \gamma_{i} ) \right|
\leq C \xi^3/(Nt),
\eeqn
given that $t\leq \min\{q/N, N^{-2/3}\}$, where we have used that $\rhofc$ is Lipschitz.
\end{proof}

\begin{proof}[Proof of Proposition \ref{p:uniHt}] By Lemma \ref{lem:Htdecomp} we need to prove universality for the matrix $\hatH_t$ defined there.  We have that
\beqn
\hatH_t = \hatA_t + ( 1 - \e^{-t} )^{1/2} GOE.
\eeqn
By Lemma \ref{lem:Htdecomp} we have that $\hatA_t \in \A$ with overwhelming probability given that $t\leq \xi/q$.   With this and Lemma \ref{lem:determbds} established, the remainder of the proof is identical to the proof of \cite[Theorem 3.1]{Sparse}.  
\end{proof}

\begin{proof}[Proof of Theorem \ref{t:uni}]
Theorem \ref{t:uni} follows from combining Proposition \ref{p:short_t_comp}, which states that the local statistics of eigenvalues of $H$ and $H_t$ are the same for $t=N^{-1+\epsilon}$ with $\epsilon\leq \beta/2$, and Proposition \ref{p:uniHt}, which states that the local statistics of eigenvalues of $H_t$ for $t=N^{-1+\epsilon}$ are the same as those of Gaussian orthogonal ensembles. Therefore the eigenvalue statistics of $H$ agree with those of Gaussian orthogonal ensembles. 
\end{proof}

\bibliography{mybib}{}
\bibliographystyle{abbrv}

\end{document}